\documentclass[12pt,reqno]{amsart}
\usepackage{hyperref}
\usepackage{color, bm, amscd, tikz-cd}
\newcommand{\cA}{{\mathcal A}}
\newcommand{\cB}{{\mathcal B}}

\newcommand{\cE}{{\mathcal E}}
\newcommand{\cF}{{\mathcal F}}
\newcommand{\cG}{{\mathcal G}}
\newcommand{\cH}{{\mathcal H}}

\newcommand{\cN}{{\mathcal N}}

\newcommand{\cR}{{\mathcal R}}
\newcommand{\cS}{{\mathcal S}}

\newcommand{\cU}{{\mathcal U}}
\newcommand{\cV}{{\mathcal V}}
\newcommand{\cW}{{\mathcal W}}

\newcommand{\bE}{{\mathbb{E}}}
\newcommand{\bT}{{\mathbb{T}}}
\newcommand{\bG}{{\mathbb{G}}}
\newcommand{\bD}{{\mathbb{D}}}
\newcommand{\bC}{{\mathbb{C}}}
\newcommand{\bS}{{\mathbb{S}}}

\newcommand{\fD}{{\mathfrak{D}}}

\newcommand{\sbm}[1]{\left[\begin{smallmatrix} #1
		\end{smallmatrix}\right]}

\newtheorem{thm}{Theorem}[section]

\newtheorem{corollary}[thm]{Corollary}
\newtheorem{lemma}[thm]{Lemma}

\newtheorem{obs}[thm]{Observation}

\newtheorem{proposition}[thm]{Proposition}

\theoremstyle{definition}

\newtheorem{definition}[thm]{Definition}
\newtheorem{remark}[thm]{Remark}

\newtheorem{example}[thm]{Example}

\numberwithin{equation}{section}

\begin{document}
\title[Distinguished varieties and the Nevanlinna--Pick problem]{Distinguished varieties and the Nevanlinna--Pick problem on the symmetrized bidisk}

\author[Das]{B. Krishna Das}
\address[Das]{Department of Mathematics, Indian Institute of Technology Bombay, Powai, Mumbai, 400076, India}
\email{dasb@math.iitb.ac.in, bata436@gmail.com}

\author[Kumar]{Poornendu Kumar}
\address[Kumar]{Department of Mathematics, University of Manitoba, Winnipeg, Manitoba, Canada
	R3T 2N2}
\email{poornendukumar@gmail.com, Poornendu.Kumar@umanitoba.ca}

\author[Sau]{Haripada Sau}
\address[Sau]{Indian Institute of Science Education and Research, Dr. Homi Bhabha Road, Pashan, Pune, Maharashtra 411008, India.}
\email{haripadasau215@gmail.com, hsau@iiserpune.ac.in}


\maketitle
\begin{abstract}
We study the uniqueness of the solutions of a solvable Pick interpolation problem in the symmetrized bidisk
$$\bG=\{(z_1+z_2,z_1z_2): z_1,z_2\in\bD\}.$$The uniqueness set is the largest set in $\bG$ where all the solutions to a solvable Pick problem coincide. There is a canonical construction of an algebraic variety, which coincides with the uniqueness set in $\bG$. The algebraic variety is called the uniqueness variety. A solvable Pick problem is called extremal if it has no solutions of supremum norm (over $\bG$) less than one. We show that if an $N$-point extremal Pick problem is such that none of the $(N-1)$-point sub-problems is extremal, then the uniqueness variety contains a distinguished variety that contains all the initial nodes. Here, a distinguished variety is an algebraic variety that intersects the domain $\bG$ and exits through its distinguished boundary. The proof of the first main result requires a thorough understanding of distinguished varieties. Indeed, this article is as much a study of the uniqueness varieties as it is about distinguished varieties. We obtain complete algebraic and geometric characterizations of distinguished varieties solving an unsettled problem left open by Pal and Shalit, J.\ Funct.\ Anal.\ 2014.

\end{abstract}
\maketitle

\section{Introduction}
Given a data set $\fD=\{(\lambda_j,w_j):1\leq j\leq N\}$, where $\lambda_1,\lambda_2,\dots,\lambda_N$ are distinct points of the open unit disk $\bD=\{\lambda\in\bC:|\lambda|<1\}$ and $w_1,w_2,\dots,w_N$ are some points in $\bD$, the original Pick problem asks {\em if there exists an analytic function $f:\bD\to\bD$ that interpolates the data, i.e.,
\begin{align}\label{PickInt}
f(\lambda_j)=w_j \quad\text{for}\quad j=1,2,\dots,N.
\end{align}}Pick showed in \cite{Pick'16} that such a function (referred to as an {\em interpolant}) exists if and only if the {\em Pick matrix}
$  \sbm{\frac{1-w_i\overline{w_j}}{1-\lambda_i\overline{\lambda_j}}}$
is positive semi-definite. Pick also showed that $f$ is unique if and only if the matrix \eqref{PickMatrix} has rank less than $N$, which is further equivalent to the existence of a Blaschke function of degree less than $N$ interpolating the data. Later in \cite{Nev19,Nev29}, Nevanlinna gave a characterization of all the interpolants. There have been numerous versions of this classical interpolation problem beginning with Abrahamse's work \cite{Abrahamse} for multiply connected domains. In the paper \cite{SarasonCLT}, Sarason introduced a new paradigm for solving the interpolation problem \eqref{PickInt} that paved the way for a several-variable generalization of the problem.  Moreover, Agler developed a new framework for analyzing the interpolation problem in the setting of the polydisk in \cite{Agler1988} -- see also the work of Ball and Trent \cite{Ball-Trent} for a simpler proof of Agler's solution and a parametrization of all interpolating functions. 

The {\em symmetrized bidisk} is the following domain in $\bC^2$:
\begin{align}\label{SymmBidisk}
\bG=\{(z_1+z_2,z_1z_2): z_1,z_2\in\bD\}.
\end{align}It is polynomially convex but not convex (not even biholormorphic to a convex domain). This domain was originally studied in \cite{AY1} in order to solve a particular case of the \textit{spectral interpolation problem} (see also \cite{BFT}). The domain is one of the few concrete domains which is interesting to both function theorists and operator theorists for its spectacularly rich function theory \cite{ ALY_MAMS, ay-jga, AY-R, BBK, BBM, BS, BK2, BK3, KZ2, KZ_JGA} and operator theory \cite{AY, ay-jot, BDSIMRN, BPSR, BK1, MSRZ, Pal-Shalit, SarkarIUMJ}. It is the first example of a domain that does not satisfy the hypothesis (that the domain be biholomorphic to a convex domain) of Lempert's Theorem but the conclusion of the theorem (that the Kobayashi and the Caratheodory metrics agree) holds, see \cite{CL}. This domain is an example of only a handful of domains where the {\em rational dilation problem} has an affirmative answer. Recently, the Nevanlinna-Pick interpolation program is carried out first in \cite{AY-R} and then in \cite{BS} in the setting of the symmetrized bidisk. Theorem \ref{T:SolveG} below describes a necessary and sufficient condition for the solvability of a given Pick problem $\fD=\{(\lambda_j, w_j):1\leq j\leq N\}$ in $\bG$, i.e., when the {\em initial nodes} $\lambda_j$ come from $\bG$ and, as in the classical case, the {\em target nodes} $w_j$ belong to $\bD$. Also see \cite[Theorem 4.2]{BS} and \cite[Theorem 5.1]{AY-R} for two other variations of necessary and sufficient conditions.

This paper assumes that a given Pick problem in $\bG$ is solvable and is concerned with the uniqueness of its solutions. We shall have use of the following terminologies.
\begin{definition}For a solvable data in $\bG$, the {\em uniqueness set} is the largest set in $\bG$ on which all the solutions agree. An interpolation problem is said to be {\em extremal}, if it is solvable but there is no interpolant of supremum norm less than one over $\bG$.
\end{definition}
If a Pick problem is not extremal, then there cannot be a unique solution. In fact, {\em the uniqueness set for a non-extremal problem is just the set of the initial nodes} - see Observation \ref{O:NonExt} for a  brief discussion. Therefore in order to investigate the uniqueness of the solutions of a Pick problem, one must start with an extremal problem. However, unlike the classical situation, an extremal Pick problem in $\bG$ may or may not have a unique solution -- see Examples \ref{this} and \ref{that} below. For the particular case of $2$-point problems, Theorem 5.4 of \cite{KZ_JGA} completely describes when the solution is unique (see also \cite{K-P}).

With the lack of uniqueness in general, it is, therefore, reasonable to study the uniqueness set. When the solution to a given Pick problem is unique, then obviously the uniqueness set is the whole of $\bG$; when the solution is not unique, the size of the uniqueness set, however,  falls by at least a dimension. More precisely, {\em for a solvable data in $\bG$, the uniqueness set coincides with an algebraic variety (i.e., the common zero set of a collection of polynomials) in $\bG$}. Moreover, for a solvable data, there is a canonical way to construct the algebraic variety - see Observation \ref{O:UV} for a brief discussion. The unique algebraic variety will be referred to as the {\em uniqueness variety}.  The driving force producing the uniqueness variety is the following result.
\begin{thm}\label{T:RatSol}
There is always a rational solution to a solvable Pick problem in $\bG$. Moreover, the rational solution can be obtained so that it is unimodular a.e.\ (with respect to the Lebesgue measure) on the symmetrized torus, i.e.,
\begin{align}\label{distG}
b\bG=\{(z_1+z_2,z_1z_2):z_1,z_2\in\bT\}.
\end{align}
\end{thm}The statement remains true even when the target data $w_j$ are square matrices; in this case, the word `unimodular' is replaced by `unitary matrices'. The proof of Theorem \ref{T:RatSol} makes crucial use of the realization theory for bounded analytic functions on $\bG$. The realization theory for $\bG$ is developed in the recent papers \cite{AY-R} and \cite{BS} via different approaches; we borrow the technique from \cite{AY-R} to prove Theorem \ref{T:RatSol}. We shall detail these topics and prove Theorem \ref{T:RatSol} in Section \ref{S:RationalFunctions}.

A certain type of algebraic varieties - starting with the work of Rudin \cite{RudinDist} - has been studied by various mathematicians around the globe from the operator theoretic point of view \cite{AM_Acta, AM_Erratum, BV01, BKS-APDE, DGH_IEOT, DasSauPAMS, DJM'16, Dis-Adv, PalIUMJ}, the function theoretic point of view \cite{DGH_IEOT, JKM'12, Knese'07, DS-1} and the geometric point of view \cite{Knese-TAMS2010, Vegulla}.
\begin{definition}
Given a bounded domain $\Omega$ in $\bC^2$, a {\em $\Omega$-distinguished variety} is the zero set of a two-variable polynomial $\xi$ such that
\begin{align}\label{dist}
Z(\xi)\cap\Omega\neq\emptyset \quad\text{and}\quad Z(\xi)\cap\partial\Omega=Z(\xi)\cap b\Omega,
\end{align}where $Z(\xi)$ is the zero set of $\xi$, $\partial\Omega$ is the topological boundary of $\Omega$ and $b\Omega$ is the distinguished boundary of $\Omega$ (i.e., the $\check{\text S}$ilov boundary with respect to the uniform algebra $\cA(\Omega)$ of functions holomorphic in $\Omega$ and continuous on $\overline\Omega$ - see \cite[Chapter 9]{AW}).
\end{definition}
This terminology is due to Agler and McCarthy \cite{AM_Acta}, where they were inexorably led to $\bD^2$-distinguished varieties while studying bivariate matrices; see also the work \cite{DS_JFA}. The distinguished boundary of $\bG$ is the symmetrized torus \eqref{distG} - see \cite[Theorem 2.4]{ay-jga}. Since symmetrized bidisk is the only domain with respect to which the distinguished varieties will be studied in this paper, we omit the mention of the domain from here on when referring to a distinguished variety with respect to $\bG$.

The first main result of this paper is the following.
\begin{thm}\label{T:SandIntro}
Let $\fD=\{(\lambda_j,w_j):1\leq j\leq N\}$ be an extremal Pick problem in $\bG$ such that none of the $(N-1)$-point subproblems is extremal. Then the uniqueness variety contains a distinguished variety that contains the initial nodes.
\end{thm}
We give two different proofs of this theorem. The first proof appeals to a similar result in \cite[Theorem 4.1]{AM_Acta}, where Agler and McCarthy proved a $\bD^2$-version of Theorem \ref{T:SandIntro}. It is natural for readers who are aware of this work to wonder if there is a way to establish Theorem \ref{T:SandIntro} from its $\bD^2$-analogue. The first question that occurs is whether an extremal Pick problem in $\bG$ remains extremal when pulled back to $\bD^2$ via the symmetrization map 
$$
\pi:(z_1,z_2)\mapsto (z_1+z_2,z_1z_2).$$ 
The answer to this question is negative, in general. For example, consider the Pick data $\{((0,0),0),((0,1/2),1/2)\}$ in $\bG$. We shall see in Example \ref{E:Ill} that this problem is extremal. But note that none of the $\bD^2$ Pick problems
 $$
(0,0)\to 0,\;(-\frac{1}{\sqrt{2}}i, \frac{1}{\sqrt{2}}i)\to\frac{1}{2},\quad\quad(0,0)\to 0,\;( \frac{1}{\sqrt{2}}i,- \frac{1}{\sqrt{2}}i)\to\frac{1}{2}$$is extremal because for example the functions $(z_1,z_2)\mapsto \pm \frac{1}{\sqrt{2}}iz_1$ solve the data, respectively and the supremum norm of these two functions are less than one. What is interesting is that the $3$-point $\bD^2$ problem
$$
(0,0)\to 0,\;(-\frac{1}{\sqrt{2}}i, \frac{1}{\sqrt{2}}i)\to\frac{1}{2},\;( \frac{1}{\sqrt{2}}i,- \frac{1}{\sqrt{2}}i)\to\frac{1}{2}
$$is extremal! Indeed, Lemma \ref{L:Added} shows that given an extremal $N$-point Pick problem in $\bG$, one can always obtain an extremal pulled-back problem in $\bD^2$ with a possibly increased number of data. This is the point of departure for the first proof.

With the already rich theory of the symmetrized bidisk and with the nuances in the first proof (that we hope to illustrate in \S \ref{S:TwoProofs}), one desires an independent proof of Theorem \ref{T:SandIntro}. This is what the second proof is about. It could not be done without a thorough understanding of the geometry of the distinguished varieties. Indeed, this paper is as much a study of the uniqueness varieties corresponding to solvable Pick problems as it is about the distinguished varieties.

A {\em numerical contraction} $F$ is a linear operator acting on a Hilbert space $\cH$ whose {\em numerical radius} is not greater than one, i.e.,
\begin{align}\label{NumericRad}
\nu(F):=\sup\{|\langle Fh,h\rangle|: \|h\|=1\}\leq 1.
\end{align}In a remarkable extension of the seminal work \cite{AM_Acta}, Pal and Shalit showed in \cite{Pal-Shalit} that {\em if a $d\times d$ matrix $F$ is a strict numerical contraction (i.e., $\nu(F)<1$), then
\begin{align}\label{DistVarG}
\cW_F:=\{(s,p)\in\bC^2:\det(F^*+pF-sI)=0\}
\end{align}is a distinguished variety.} What about the $\nu(F)=1$ case? Simple examples such as $F=\sbm{0&2\\0&0}$ and $F=\sbm{1&0\\0&1}$ show that $\cW_F$ may or may not be a distinguished variety when $\nu(F)=1$. The mystery grows as one learns the converse direction: {\em every distinguished variety is of the form \eqref{DistVarG} for some matrix $F$ with $\nu(F)\leq 1$ (not strict).} See Theorem 3.5 of \cite{Pal-Shalit} for this result. This rather fortuitous connection between numerical contractions and distinguished varieties stems from an earlier work \cite{BPSR}, where the authors introduced a tool whose contribution to the theory of symmetrized bidisk has been extraordinary.  In an effort to better understand this connection, the recent work \cite{BKS-APDE} showed that in the above characterization of distinguished varieties, the numerical contraction $F$ can be chosen (in both directions) from a considerably smaller class of numerical contractions, viz.,
$$
\{PU+U^*P^\perp: P \text{ an orthogonal projection, }U \text{ a unitary acting on }\bC^d, d\geq1\}.
$$See \cite[Lemmas 3.10 and 6.2]{BKS-APDE} to see that the above is a proper subclass of numerical contractions and \cite[Theorem 6.1]{BKS-APDE} for the refinement of the Pal--Shalit result. The mystery however remained: examples of projections $P$ and unitaries $U$ can be found to show that $\cW_{PU+U^*P^\perp}$ may or may not be distinguished when $\nu(PU+U^*P^\perp)=1$, see Examples 6.4 and 6.5 in \cite{BKS-APDE}.

This paper resolves the issue. If an operator has no reducing subspace where it is unitary, then it is called {\em completely non-unitary (c.n.u.)}. In this paper, the notation $M_d(\bC)$ is used to denote the algebra of $d\times d$ complex matrices.
\begin{thm}\label{T:IG}
Let $F$ in $M_d(\bC)$ be a numerical contraction (not necessarily strict). Then $\cW_F$ as in \eqref{DistVarG} is a distinguished variety if and only if $F$ is completely non-unitary. Moreover, each irreducible component of $\cW_F$ intersects $\bG$, when $F$ is a c.n.u.\ numerical contraction. Conversely, every distinguished variety with each irreducible component intersecting $\bG$ is of the form \eqref{DistVarG} for a c.n.u.\ numerical contraction $F$.
\end{thm}
This result is contained in Theorem \ref{T:DistMain} - the second main result of this paper -  that exhibits some other geometric equivalent conditions for $\cW_F$ to be a distinguished variety. This result in turn gives a new proof of Proposition 4.1 of \cite{Knese-TAMS2010}, where Knese produces an elegant geometric characterization of distinguished varieties with respect to $\bD^2$.

We then characterize the pairs $(P,U)$ of orthogonal projections $P$ and unitary $U$ acting on finite-dimensional Hilbert spaces so that $PU+U^*P^\perp$ is completely non-unitary. This in turn gives a $(P,U)$-version of Theorem \ref{T:IG} -  see Theorem \ref{T:PUversion}.

Finally, it is a pleasure to thank Professor John E.\ McCarthy profusely for his generous help in understanding some perturbation theory for matrices that plays a crucial role in the proof of Theorem \ref{T:SandIntro}. The authors also thank Professor Tirthankar Bhattacharyya for some insightful discussions at the beginning of this project.

\section{Distinguished varieties and numerical contractions}\label{S:DistVar}
The purpose of this section is to explore the connection between numerical contractions and distinguished varieties. We begin with the following well-known result about the numerical range of a matrix $F$, i.e., the set
 $$
 W(F):=\{\langle Fh,h\rangle: \|h\|=1 \}.
 $$
\begin{thm}[See Theorem 5.1-9 in \cite{GustavsonRao}]\label{T:BndryEgenvle}
If an eigenvalue $\beta$ of a $d\times d$ matrix $F$ belongs to $\partial W(F)$, the topological boundary of $W(F)$, then $F$ is unitarily equivalent to $\sbm{\beta I_{\bC^r}&0\\0& F'}$, $r\leq d$ such that $\beta$ is not an eigenvalue of $F'$.
\end{thm}
This result has an important consequence: Consider the spectrum $\sigma(F)$ of a numerical contraction $F$. If it contains a unimodular member $\beta_1$, then it must belong to the boundary of $W(F)$ and therefore by Theorem \ref{T:BndryEgenvle}, $F$ must be unitarily equivalent to a matrix of the form $\sbm{\beta_1I_{\bC^{r_1}}&0\\0& F_1}$. Now consider $\sigma(F_1)$ and repeat the process until we arrive at a matrix $F'$ whose spectrum does not contain any unimodular member. Note that $F'$ would then obviously be a c.n.u.\ numerical contraction. Consequently, we have the following result, which is perhaps well-known.
\begin{proposition}\label{Obs:Decomposition}
  Every numerical contraction acting on a finite-dimensional Hilbert space decomposes into a direct sum of a unitary and a c.n.u.\ numerical contraction.
\end{proposition}

For a square matrix $F$ and a complex number $p$, we are interested in the linear pencil $F^*+pF$. We begin with a couple of elementary facts about these linear pencils each of which makes a crucial contribution to the proof of the main result of this section.
\begin{lemma}\label{L:Auxiliary}
Let $F$ be a $d\times d$ matrix and $p\in\mathbb C$ be such that $|p|\neq 1$.
\begin{itemize}
\item[(1)]For $0\leq r\leq d$, $F=\sbm{\alpha I_{\mathbb C^r} & 0\\ 0 & A}$ if and only if
$F^* + p F=\sbm{\mu I_{\mathbb C^r} & 0\\0 & B}$;
\item[(2)]If $F$ is a numerical contraction, then
$$W(F^* + p F)\cap \{\beta + \bar{\beta}p: |\beta |=1\} \subseteq \partial W(F^* + p F).$$
\end{itemize}
\end{lemma}
\begin{proof}[Proof of (1)] The `only if' part is obvious. For the other part,
let us write out $F$ as  a block $2\times 2$ matrix $\begin{bmatrix}
		F_{11} & F_{12}\\
		F_{21} & F_{22}
\end{bmatrix}$
making use of the same basis with respect to which we have assumed that
$F^*+pF=\begin{bmatrix}
	\mu I_{\bC^r} & 0\\
	0 & B
\end{bmatrix}$. Then we get
$$
\begin{bmatrix}
	F_{11}^* & F_{21}^*\\
	F_{12}^* & F_{22}^*
\end{bmatrix}+p\begin{bmatrix}
F_{11} & F_{12}\\
F_{21} & F_{22}
\end{bmatrix}=\begin{bmatrix}
\mu I_{\bC^r} & 0\\
0 & B
\end{bmatrix}. $$
From the equality of the off-diagonal entries we get
$$F_{21}^*+pF_{12}=0, \quad F_{12}^*+pF_{21}=0,$$
and hence 
$$F_{21}^*=-pF_{12}=-p\left(-\overline{p}F_{21}^*\right)=|p|^2F_{21}^*.$$
Then the assumption that $p$ is not unimodular forces $F_{21}=0$. A similar arguments gives $F_{21}=0$. 

\noindent{\em Proof of (2).} It is enough to show that when $\beta$ is unimodular, any neighbourhood of $\beta+\overline\beta p$ contains a point that is not in $W(F^*+pF)$. Let us first note that 
$$
W(F^*+pF)=\{\gamma+\overline{\gamma}p: \gamma\in W(F^*)\subset \overline{\mathbb D}\}.
$$ For every $\epsilon > 0$, the point $(1+\epsilon/2)(\beta + \overline{\beta} p)$ lies in the $\epsilon$-ball centered at $\beta + \overline{\beta} p$. We claim this point cannot belong to $W(F^* + pF)$. If it did, there would exist some $\alpha \in W(F^*)\subset\overline{\mathbb{D}}$ satisfying 
	\[
	(1+\epsilon/2)(\beta + \overline{\beta} p) = \alpha + \overline{\alpha} p.
	\]
The above equality is the same as $(\overline{\alpha}-\overline\beta-\overline\beta\epsilon/2)p=\beta+\beta\epsilon/2 -\alpha$. The number $\overline{\alpha}-\overline\beta-\overline\beta\epsilon/2$ is non-zero because $|\beta|=1$ and $|\alpha|\leq1$. This forces $p$ to be unimodular, contrary to our assumption. This completes the proof.
\end{proof}

The closure of the symmetrized bidisk $\bG$ is denoted by $\Gamma$, i.e.,
$$
\Gamma:=\{(z_1+z_2,z_1z_2):z_1,z_2\in\overline{\bD}\}.
$$Among many characterizations of $\bG$, the following are of particular importance.
\begin{thm}[See Theorem 2.1 in \cite{ay-jga}]\label{T:Charc_G}
The following are equivalent for $s,p$ in $\bC$:
\begin{enumerate}
\item[(i)]  $(s,p) \in \bG$;
\item[(ii)]  $|s-\overline s p| < 1-  |p|^2$;
\item[(iii)] $|p|<1$ and there exists a unique $\beta\in\bD$ such that $s=\beta+\overline \beta p$. Moreover, $\beta$ is given by $\beta=\frac{s-\overline s p}{1-|p|^2};$
\item[(iv)] $|s|<2$ and for all $\alpha\in\overline{\bD}$, $\left|\frac{2\alpha p-s}{2-\alpha s}\right|<1$.
\end{enumerate}
\end{thm}
The above theorem, especially part (iii), will be used throughout the paper. The following is the main result of this section.
\begin{thm}\label{T:DistMain}
For a numerical contraction $F$ in $M_d(\bC)$, the set
$$
\cW_F=\{(s,p)\in\bC^2: \det(F^*+pF-sI)=0\}=\{(s,p)\in\bC^2:(s,p)\in\sigma(F^*+pF,pI)\}
$$is an algebraic variety that always intersects $\Gamma$. Moreover, the following are equivalent:
\begin{itemize}
\item[(1)]$\mathcal W_F$ is a distinguished variety;
\item[(2)] $F$ is completely non-unitary, or equivalently, $\sigma(F)\cap\bT=\emptyset$;
\item[(3)] $\cW_F\subset\pi(\bD^2\cup\bT^2\cup\bE^2)$, where $\bE=\bC\setminus \overline\bD$;
\item[(4)] Every irreducible component of $\cW_F$ intersects $\bG$.
\end{itemize}
Conversely, if $\cW$ is any distinguished variety such that each of its irreducible components meets $\bG$, then there exists a c.n.u.\ numerical contraction $F$ such that $\cW=\cW_F$.
\end{thm}
We make a couple of remarks before we prove this theorem.
\begin{remark}
  Let $F$ in $M_d(\bC)$ be any numerical contraction. Then by Proposition \ref{Obs:Decomposition}, $F=F_0\oplus F_u$, where $F_0$ is a completely non-unitary matrix and $F_u$ is a unitary matrix. It is easy to note that $\cW_F=\cW_{F_0}\cup\cW_{F_u}$ and that for some $0\leq r\leq d$,
  $$
  \cW_{F_u}=\bigcup_{j=1}^r\{(\beta_j+\overline{\beta_j}p,p):\beta_j\in\bT,p\in\bC\}.
  $$ Therefore in view of part (iii) of Theorem \ref{T:Charc_G}, we have $\cW_F\cap\bG=\cW_{F_0}\cap\bG$.
\end{remark}

\begin{remark}
One of the important consequences of Theorem~\ref{T:DistMain} is that if $\cW$ is a distinguished variety with an irreducible component not intersecting $\bG$, then there is no numerical contraction $F$ such that $\cW=\cW_F$.
A concrete example demonstrating this phenomenon is the variety
$$
\cW=\{(s,p)\in\bC^2:(s^2-4p)(p-1)=0\}.
$$This is a distinguished variety. The irreducible component $\{(s,p)\in\bC^2:p-1=0\}$ does not intersect $\bG$. There is, in fact, no matrix $F$ such that $\cW=\cW_F$. Because if there was an $F$, then in view of the joint spectrum representation of $\cW_F$:
$$
\cW_F=\{(s,p)\in\bC^2:(s,p)\in\sigma(F^*+pF,pI)\},
$$ $\sigma(F^*+F)$ would have to contain the set $\{z_1+z_2:z_1z_2=1\}$. This is absurd because the spectrum of a matrix is a finite set.
\end{remark}
\begin{proof}[Proof of Theorem \ref{T:DistMain}]
That $\cW_F$ is an algebraic variety follows from the fact that the determinant is always a polynomial; for this, $F$ does not even have to be a numerical contraction. To see that $\cW_F$ intersects $\Gamma$, note that $\cW_F$ contains the points $(\sigma(F^*),0)$, which is contained in $\Gamma$ because $\sigma(F^*)\subset \overline\bD$.

For the moreover part, we prove $(1)\Rightarrow(2)\Rightarrow(3)\Rightarrow(4)\Rightarrow(1)$. The equivalence stated in item $(2)$ is a consequence of Theorem \ref{T:BndryEgenvle}. Indeed, if a c.n.u.\ numerical contraction $F$ has a unimodular eigenvalue, then Theorem \ref{T:BndryEgenvle} shows that $F$ must decompose into a matrix of the form $\beta I_{\bC^r}\oplus F'$ for some $1\leq r$. This is not possible because $F$ is a c.n.u.\ numerical contraction. The other direction is straightforward.

{\em Proof of $(1)\Rightarrow(2)$}: This is easy to see. Indeed, if there was a unitary $U$ such that $F$ is unitarily equivalent to $U\oplus F'$, then $\cW_F$ would contain points of the form $(\beta+\overline{\beta} p,p)$ such that $(\beta,p)\in\bT\times\bD$, which would violate the distinguished property of $\cW_F$.

{\em Proof of $(2)\Rightarrow(3)$}: This is the non-trivial part of the proof and has some important implications that we shall record as we progress with the proof.

To establish the containment in $(3)$, it is equivalent to show that $\cW_F$ does not intersect
$$
R_1:=\pi(\bD\times\bT)\cup\pi(\bT\times\bD)\cup\pi(\bE\times\bT)\cup\pi(\bT\times\bE)
$$and
$$
R_2:=\pi(\bD\times\bE)\cup\pi(\bE\times\bD).
$$Let us suppose on the contrary that $(s,p)=\pi(z_1,z_2)\in\cW_F\cap \pi(\bD\times\bE)$. This means that there exists a norm-unit vector $v$ such that with $\beta=\langle F^*v,v\rangle $,
$$
z_1+z_2=\beta+\overline\beta z_1z_2,
$$ or equivalently,
$$
z_2=\frac{\beta-z_1}{1-\overline\beta z_1}.
$$Since $F$ is a numerical contraction, $\beta\in\overline{\bD}$, and since $z_1\in\bD$, the above representation of $z_2$ implies that it must belong to $\overline{\bD}$ contradicting the assumption that $z_2\in\bE$. A similar argument works for the case $\pi(\bE\times\bD)$. Thus $\cW_F$ does not intersect $R_2$. Note that we have not used the completely non-unitary property of $F$ to rule out $R_2$; the analysis works for any numerical contractions.

To show that $\cW_F$ does not intersect $R_1$, we begin with the following characterization of $R_1$.
\begin{lemma}\label{L:R1}
A pair $(s,p)$ is in $R_1$ if and only if $p$ is not in $\bT$ and $\frac{s-\overline s p}{1-|p|^2}$ is in $\bT$.
\end{lemma}
\begin{proof}
If $(s,p)=(\zeta+z,\zeta z)\in R_1$ such that $\zeta\in\bT$ and $|z|\neq1$, then obviously $|p|\neq 1$. Straightforward computation yields $s-\overline s p= \zeta(1-|p|^2)$, proving the forward assertion. The converse follows immediately from the equality
$$
s=\frac{s-\overline s p}{1-|p|^2}+\frac{\overline s- s\overline p}{1-|p|^2}p.
$$
\end{proof}
Now let us suppose on the contrary that $(s,p)\in \cW_F\cap R_1$. Then there exists a norm-unit vector $v$ such that $(F^*+pF)v=sv$. Thus $s$ is an eigenvalue of $F^*+pF$, and with $\beta=\langle F^*v,v\rangle$ we have $s=\beta+\overline{\beta} p$. Since $|p|\neq 1$, we must have $\beta=\frac{s-\overline s p}{1-|p|^2}$. By Lemma \ref{L:R1}, $\beta\in\bT$. Thus by part (2) of Lemma \ref{L:Auxiliary}, $s=\beta+\overline{\beta} p$ must belong to $\partial W(F^*+pF)$. Since $s$ is an eigenvalue of $F^*+pF$, by Theorem \ref{T:BndryEgenvle}, there exists a positive integer $r\leq d$, and a unitary $\tau$ such that
$$
\tau(F^*+pF)\tau=
\begin{bmatrix}
sI_{\bC^r}&0\\0& G
\end{bmatrix}
$$for some matrix $G$ whose spectrum does not contain $s$. It turns out that we do not need the spectrum assertion on $G$. Now apply part (1) of Lemma \ref{L:Auxiliary} for $\tau^*F\tau$ to obtain
\begin{align}\label{gradedF}
\tau^*F\tau=
\begin{bmatrix}
\overline{\beta} I_{\bC^r}&0\\0& F'
\end{bmatrix}
\end{align}for some matrix $F'$. The scalar in the $(11)$-entry of $\tau^*F\tau$ has to be $\overline{\beta}$ because $s=\beta+\overline{\beta} p$. But the fact that $\beta\in\bT$ contradicts the hypothesis that $F$ is completely non-unitary. This completes the proof of $(2)\Rightarrow(3)$.

{\em Proof of $(3)\Rightarrow(4)$}: This follows from the fact that no algebraic variety can be entirely contained in $\bT^2\cup\bE^2$. Indeed, if $Z(\xi)\subset \bT^2\cup\bE^2$ for a polynomial $\xi$ in $C[z_1,z_2]$, then $\xi$ cannot be a polynomial of only one of the variables. Now choose $z_1\in\bD$ and consider the polynomial $\xi_{z_1}(z_2)=\xi(z_1,z_2)$. Since $\bC$ is an algebraically closed field, $\xi_{z_1}$ must vanish at some point. This violates the containment $Z(\xi)\subset \bT^2\cup\bE^2$. Now the proof follows from the understanding that a variety in $(z_1+z_2,z_1z_2)$-coordinate is the symmetrization of a variety in $(z_1,z_2)$-coordinate - see Lemma 3.1 in \cite{Pal-Shalit}.

{\em Proof of $(4)\Rightarrow(1)$}: If every irreducible component of $\cW_F$ meets $\bG$, then clearly $F$ must be completely non-unitary. We have already proved that in such a case, $\cW_F$ is contained in $\pi(\bD^2\cup\bT^2\cup\bE^2)$. This, together with the fact that no algebraic variety is entirely contained in $\bT^2\cup\bE^2$ ensures the non-emptiness of $\cW_F\cap\bG$. The rest follows from the definition of a distinguished variety.

For the converse part, it is enough to assume that $\cW$ is an irreducible variety because if it is not, then one considers the direct sum of the numerical contractions obtained for each irreducible component.

So given an irreducible distinguished variety $\cW$, we have to find a c.n.u.\ numerical contraction $F$ such that $\cW=\cW_F$. The idea goes back to the work \cite{AM_Acta} of Agler and McCarthy, which was adopted in the setting of the symmetrized bidisk to obtain an $F$ by Pal and Shalit \cite[Theorem 3.5]{Pal-Shalit}. We sketch the proof for completeness. The first step is to construct a Hardy-type reproducing kernel Hilbert space $H(k)$ on $\cW\cap \bG$ that originates from a regular Borel measure on the boundary $\cW\cap b\bG$. The existence of the measure is guaranteed by a result of the theory of Riemann surfaces. Consider the pair $(M_s,M_p)$ of multiplication by coordinate functions on $H(k)$. It turns out (see Lemma 3.4 in \cite{Pal-Shalit}) that the closed symmetrized bidisk $\Gamma$ is a spectral set for $(M_s,M_p)$ (the so called $\Gamma$-contraction, a terminology due to Agler and Young \cite{ay-jot}), i.e.,
$$
\|f(M_s,M_p)\|\leq \sup\{|f(s,p)|:(s,p)\in\Gamma\}
$$for every polynomial $f$ in two variables. Note that this makes $M_p$ a contraction. Invoke Theorem 4.2 in \cite{BPSR} that says that there is a unique numerical contraction $F:\operatorname{Range}(I-M_pM_p^*)\to\operatorname{Range}(I-M_pM_p^*)$ such that
$$
M_s^*-M_sM_p^*=(I-M_pM_p^*)^{\frac{1}{2}}F(I-M_pM_p^*)^{\frac{1}{2}}.
$$Pal and Shalit showed that $\operatorname{Range}(I-M_pM_p^*)$ is finite dimensional (thus $F$ is a matrix) and that $\cW\cap\bG=\cW_F\cap\bG$.

While there is no reason to believe that $\nu(F)<1$, we note that $F$ can be replaced by its c.n.u.\ part harmlessly. This is because the unitary part of $F$ has no contribution in the intersection $\cW_F\cap\bG$. Now, we know that if $F$ is c.n.u., then every irreducible component of $\cW_F$ meets $\bG$. Since $\cW\cap\bG=\cW_F\cap\bG$, and $\cW$ is irreducible, $\cW_F$ must be irreducible. So we have arrived at a situation where two irreducible varieties coincide in $\bG$. This forces the two varieties to be the same. This completes the proof.
\end{proof}
During the course of the proof of $(2)\Rightarrow(3)$, there emerged an interesting fact that we record below as a proposition. Note that to arrive at the conclusion \eqref{gradedF}, all we used is that $F$ is a numerical contraction and $(s,p)\in\cW_F$ such that $|p|\neq 1$ and $\frac{s-\overline s p}{1-|p|^2}\in\bT$. Thus we have already proved $(1)\Rightarrow(2)$ of the following result; the implications $(2)\Rightarrow(3)$ and $(3)\Rightarrow(1)$ are obvious.
\begin{proposition}\label{P:sheet}
Let $F\in M_d(\mathbb C)$ be a numerical contraction and $\beta\in\bT$. Then the following are equivalent:
\begin{itemize}
\item[(1)]
$\mathcal W_F$ contains a point of the form $(\beta+\overline{\beta} p,p)$ such that $|p|\neq 1$;
\item[(2)] $F$ is unitarily equivalent to $\begin{bmatrix} \overline \beta I_{\mathbb C^r} &0\\ 0 & F'
\end{bmatrix}$ for some matrix $F'$;
\item[(3)] $\mathcal W_{\beta}\subseteq \mathcal W_F$.
\end{itemize}
\end{proposition}

\begin{remark}
We believe that items (2) and (3) of Proposition \ref{P:sheet} are equivalent even in the case when $|\beta|\neq1$. This is easily seen to be true when $F$ is normal for example. Indeed, the containment $\cW_\beta\subset\cW_F$ implies that
\begin{align}\label{ZeroinC}
\det[(F^*-\beta I)+p(F-\overline\beta I)]=0 \quad\text{for all }p\in\bC.
\end{align}If $F$ is normal, then so is $F^*-\beta I$ and hence is diagonalizable, which readily implies that $F^*-\beta I$ must have $0$ as an eigenvalue. This establishes item (2). Note that in this case, we used the information \eqref{ZeroinC} just for $p=0$.
\end{remark}
\begin{remark}
If an algebraic variety $\cV$ is contained in $\bD^2\cup\bT^2\cup\bE^2$, then as explained in the proof of $(3)\Rightarrow(4)$ of Theorem \ref{T:DistMain}, every irreducible component of $\cV$ meets $\bD^2$. In Proposition 4.1 of \cite{Knese-TAMS2010}, Knese elegantly showed that the converse is also true: if $\cV$ is a distinguished variety with respect to $\bD^2$ each of whose irreducible components meets $\bD^2$, then $\cV\subset\bD^2\cup\bT^2\cup\bE^2$. Theorem \ref{T:DistMain} gives a new proof of this result. Indeed, let $\cV$ be a distinguished variety with respect to $\bD^2$ such that each irreducible component intersects $\bD^2$. Then $\cW=\pi(\cV)$ is a distinguished variety with each irreducible component intersecting $\bG$. Apply the converse of Theorem \ref{T:DistMain} to conclude that $\cW=\cW_F$ for some c.n.u.\ numerical contraction $F$. Then item (3) of Theorem \ref{T:DistMain} implies that $\pi(\cV)\subset\pi(\bD^2\cup\bT^2\cup\bE^2)$. Now pull back this containment to obtain Knese's characterization.
\end{remark}

\begin{definition}
A set $X\subset\bC^2$ is said to have the {\em distinguished property}, if its closure satisfies the following boundary constraint:
$$\overline{X}\cap\partial \Gamma=\overline{X}\cap b\bG.$$
\end{definition}Therefore by definition, a distinguished variety has the distinguished property. Such sets have appeared before in \cite{DGH_IEOT} while studying the totally Abelian analytic Toeplitz operators.

Given a numerical contraction $F$, an interesting consequence of Theorem \ref{T:DistMain} is that the intersection $\cW_F\cap\bG$ always enjoys the distinguished property, while the variety $\cW_F$ need not have the distinguished property (precisely when $F$ is not c.n.u.).
\begin{corollary}
Given a numerical contraction $F$, the set $\cW_F\cap\bG$ always satisfies the distinguished property.
\end{corollary}
\begin{proof}
Note that if $F'$ is the c.n.u.\ part of $F$, then $\cW_F\cap\bG=\cW_{F'}\cap\bG$ because the component corresponding to the unitary part never intersects the open symmetrized bidisk. Since $\cW_{F'}$ is a distinguished variety and each of its irreducible components meets $\bG$, we have
$$
\overline{\cW_{F'}\cap\bG}\cap \partial\Gamma=\cW_{F'}\cap \partial\Gamma=\cW_{F'}\cap b\bG=\overline{\cW_{F'}\cap\bG}\cap b\bG.
$$\end{proof}
The royal variety $\mathcal R=\{(s,p)\in\bC^2: s^2-4p=0\}$ is a special distinguished variety. It has proved to be of great importance in understanding the complex geometry of $\bG$, see for example \cite[Theorem 1.4]{ALY_MAMS} and \cite[Theorem 2.6]{BBM}.  The mystery surrounding the $\nu(F)=1$ case (discussed in the Introduction following equation \eqref{DistVarG}) led to the intuitive expectation that for $\cW_F$ to be a distinguished variety when $\nu(F)=1$, it must contain the royal variety. The following example shows that this need not be true.
 \begin{example}
 Let $F=\sbm{
 1/2 & 1\\
 0 & 1/2}$. Then a simple calculation reveals that $\nu(F)=1$. Also, note that
 \[\mathcal W_F=\{(s,p)\in \mathbb C^2: ((1+p)-2s)^2-4p=0\}.
 \]
 By part (2) of Theorem~\ref{T:DistMain}, $\mathcal W_F$ is a distinguished variety. But $\mathcal W_F$ does not contain the royal variety $\cR$ because for example $(0,0)$ is in $\cR$ but not in $\cW_F$.
 \end{example}

As elaborated in \cite[Lemma 3.1]{Pal-Shalit}, $\bG$-distinguished varieties are pecisely the image of $\bD^2$-distinguished varieties under the symmetrization map $\pi$. Theorem 3.2 in \cite{BKS-APDE} describes the $\bD^2$-distinguished varieties as
$$
\cW_{P,U}=\{(z_1,z_2)\in\bC^2: (z_1,z_2)\in\sigma(P^\perp U+z_1z_2 PU, U^*P+z_1z_2U^*P^\perp)\}
$$where $P$ is an orthogonal projection and $U$ a unitary matrix acting on $\bC^d$, for some $d\geq1$. The ingedients $(P,U)$ arise from an application of \cite[Theorem 3.1]{BCL} which characterises commuting isometries. It was detailed in the proof of Theorem 6.1 in \cite{BKS-APDE} that $\pi(\cW_{P,U})=\cW_{PU+U^*P^\perp}$  ($\cW_F$, for a matrix $F$, is as in Theorem \ref{T:DistMain}). As mentioned in the Introduction (see Lemma 3.10 of \cite{BKS-APDE}) that
\begin{align*}
\cN=\{PU+U^*P^\perp: P \text{ is projection, } U \text{ a unitary on }\bC^d\}
\end{align*} is a family of numerical contractions. Thus what we have is the following theorem.
\begin{thm}[Theorem 6.1 in \cite{BKS-APDE}]\label{T:BKP_APDE}
Let $P$ be an orthogonal projection and $U$ be a unitary acting on a finite-dimensional Hilbert space such that $\nu(PU+U^*P^\perp)<1$. Then $\cW_{PU+U^*P^\perp}$ is a distinguished variety. Conversely, if $\mathcal W$ is any distinguished variety, then there is an orthogonal projection $P$ and a unitary $U$ acting on a finite-dimensional Hilbert space such that $\mathcal W\cap\mathbb G=\mathcal W_{PU+U^*P^\perp}\cap\mathbb G$.
\end{thm}
It is therefore natural to ask for a $(P,U)$-version of Theorem \ref{T:DistMain}. This requires the knowledge of exactly when a numerical contraction of the form $PU+U^*P^\perp$ is completely non-unitary. The following lemma provides it.
\begin{lemma}
  For an orthogonal projection $P$ and a unitary $U$ acting on $\bC^d$, $PU+U^*P^\perp$ is not c.n.u.\ if and only if there is a non-trivial $U$-reducing subspace contained either in $\operatorname{Range}P$ or in $\operatorname{Range}P^\perp$.

\end{lemma}
\begin{proof}
  Suppose that there is a non-trivial subspace $\cH\subset\operatorname{Range}P$ such that
$U=\sbm{U'&0\\0&W}$ with respect to the decomposition $\sbm{\cH \\ \cH^\perp}$. The case $\cH\subset\operatorname{Range}P^\perp$ can be dealt with similarly. Since $\cH\subset\operatorname{Range}P$, the projections $P,P^\perp$ take the form $P=\sbm{I_\cH&0\\0&Q}$ and $P^\perp=\sbm{0&0\\0&Q^\perp}$ for some orthogonal projection $Q$ on $\cH^\perp$. It is now a plain matrix computation to see that $PU+U^*P^\perp=\sbm{U'&0\\0&QW+W^*Q^\perp}$. Since $\cH$ is non-trivial and $U'$ is unitary, $PU+U^*P^\perp$ cannot be completely non-unitary. Note that this direction works for infinite dimensional Hilbert spaces also. We shall use finite dimensionality in the other direction.

Note that if
  $$
  U=\begin{bmatrix}
    A&B\\C&D
  \end{bmatrix}:\begin{bmatrix}
    \operatorname{Range}P\\ \operatorname{Range}P^\perp
  \end{bmatrix}\to\begin{bmatrix}
    \operatorname{Range}P\\ \operatorname{Range}P^\perp
  \end{bmatrix},
  $$then
  $$
  PU+U^*P^\perp=
  \begin{bmatrix}
  A&B+C^*\\0&D^*
  \end{bmatrix}:\begin{bmatrix}
    \operatorname{Range}P\\\operatorname{Range}P^\perp
  \end{bmatrix}\to\begin{bmatrix}
    \operatorname{Range}P\\ \operatorname{Range}P^\perp
  \end{bmatrix}.
  $$If $PU+U^*P^\perp$ is not c.n.u., then it has an eigenvalue, say $\zeta$ on $\bT$, the unit circle. But from the above matrix structure of $PU+U^*P^\perp$, we see that $\sigma(PU+U^*P^\perp)=\sigma(A)\cup\sigma(D^*)$. If $\zeta\in\sigma(A)$, then there is a non-trivial space $\cH\subset\operatorname{Range}P$ such that $A=\sbm{\zeta I_{\dim\cH}&0\\0&A'}$. The space $\cH$ clearly reduces $U$. If $\zeta\in\sigma(D^*)$, then we shall get a non-trivial $U$-reducing subspace contained in $\operatorname{Range}P^\perp$.
\end{proof}
The following result is now an easy consequence of Theorem \ref{T:DistMain}, Theorem \ref{T:BKP_APDE} and the fact that $PU+U^*P^\perp$ is a numerical contraction.
\begin{thm}\label{T:PUversion}
For an orthogonal projection $P$ and a unitary $U$ acting on $\bC^d$, the set $\cW_{PU+U^*P^\perp}$ is an algebraic variety that always intersects $\Gamma$. Moreover, the following are equivalent:
\begin{itemize}
\item[(1)]$\mathcal W_{PU+U^*P^\perp}$ is a distinguished variety;
\item[(2)] $PU+U^*P^\perp$ is completely non-unitary, or equivalently, $\sigma(PU+U^*P^\perp)\cap\bT=\emptyset$;
\item[(3)] There is no non-trivial $U$-reducing subspace contained either in $\operatorname{Range}P$ or in $\operatorname{Range}P^\perp$;
\item[(4)] $\cW_{PU+U^*P^\perp}\subset\pi(\bD^2\cup\bT^2\cup\bE^2)$;
\item[(5)] Every irreducible component of $\cW_{PU+U^*P^\perp}$ intersects $\bG$.
\end{itemize}
Conversely, if $W$ is any distinguished variety such that each of its irreducible components meets $\bG$, then there exist a projection $P$ and a unitary $U$ such that $PU+U^*P^\perp$ is c.n.u.\ and $\cW=\cW_{PU+U^*P^\perp}$.
\end{thm}

\section{Rational inner functions on the symmetrized bidisk}\label{S:RationalFunctions}
If $\Psi$ on $\bD^2$ is an operator-valued bounded analytic function, then it is not difficult to see that the radial limit
$$\Psi(\zeta,\eta):=\lim_{r\to 1-}\Psi(r\zeta,r\eta)$$ exists almost everywhere in $\bT^2$. Since every function on $\bG$ can be viewed as a symmetric function on $\bD^2$, it follows that every bounded analytic function on $\bG$ has a radial limit almost everywhere in $b\bG=\{(z_1+z_2,z_1z_2):z_1,z_2\in\bT\}$. An operator-valued bounded analytic function $\Psi$ on $\bG$ is said to be {\em inner}, if the radial limits are isometric operators a.e.\ in $b\bG$.  See \cite{BK3} for details discussion on inner functions on the symmetrized bidisc or more genreally on the {\em quotient domains}. 

This section aims to establish Theorem \ref{T:RatSol}, a crucial result utilized in the subsequent section. The main tool is an elegant representation formula for bounded analytic functions. The classical version of this formula states that for two Hilbert spaces $\cE$ and $\cF$, a function $\Psi:\bD\to\cB(\cE,\cF)$ is contractive analytic if and only if there is an auxiliary Hilbert space $\cH$ and a unitary operator $U=\sbm{A&B\\C&D}:\sbm{\cE\\ \cH}\to\sbm{\cF\\ \cH}$ such that
$$
\Psi(z)= A+zB(I-zD)^{-1}C.
$$This is called the realization formula for $\Psi$. This result has been extended to polydisk \cite{Agler1988} and various other generalities \cite{Ambrozie, BBF'07, BH'13, DMM-Crell'07}. The realization formula for bounded analytic functions on $\bG$ was found in \cite{AY-R} and \cite{BS} by different methods. In both approaches, the rational functions
$$
\varphi(\alpha,s,p)=\frac{2\alpha p-s}{2-\alpha s} \quad\text{for}\quad\alpha\in\bD \quad\text{and}\quad (s,p)\in\bG
$$have played a significant role. For a fixed $(s,p)\in\bG$, the function $\varphi(\cdot,s,p)$ is analytic in $\bD$ and continuous on $\overline\bD$. Thus for a contractive operator $\tau$, the operator
$$
\varphi(\tau,s,p):=(2\tau p-s)(2-\tau s)^{-1}
$$is well-defined - see for example \cite[Sec.\ 2, Ch.\ III]{Nagy-Foias}. Moreover, the supremum of $\varphi(\cdot,s,p)$ over $\bD$ is not greater than one by item (iv) of Theorem \ref{T:Charc_G}, therefore $\varphi(\tau,s,p)$ is a contraction. The realization formula in the setting of $\bG$ is the following.
\begin{thm}\label{T:RFSymmDsk}
A function $\Psi:\mathbb G\to\cB(\cE,\cF)$ is contractive analytic if and only if there exist a Hilbert space $\mathcal{H}$ and unitary operators
\begin{align}
\tau:\mathcal{H}\to \mathcal{H} \quad \text {and} \quad
U=\begin{bmatrix}
A & B\\
C & D
\end{bmatrix}
:\begin{bmatrix}\cE\\\mathcal{H}\end{bmatrix}\to \begin{bmatrix}\cF\\\mathcal{H}\end{bmatrix}
\end{align}
such that
\begin{align}
\Psi(s,p)= A+B\varphi(\tau,s,p)(I-D\varphi(\tau,s,p))^{-1}C.
\end{align}
\end{thm}
We show that a rational inner function on the symmetrized bidisk has a {\em finite dimensional realization formula}, i.e., when the auxiliary Hilbert space $\cH$ in Theorem \ref{T:RFSymmDsk} is of finite dimension. Such a result is known for the bidisk, which we state below. See the papers \cite{Ball-Sad-Vin, Bickel-Knese, Knese, Kummert} for various proofs and generalizations as well as \cite{ABJK} for a recent application of this result in proving Carath\'eodory approximation result. 
\begin{thm}\label{T:BDisk}
Let $\Psi$ be a $d\times d$ matrix-valued analytic function on $\mathbb D^2$. Then the following are equivalent:
\begin{enumerate}
\item[(\text{RI})] $\Psi$ is a rational inner function;
\item[(\text{AD})]There exist analytic functions $F_1,F_2:\mathbb D^2\to \cB(\mathbb C^d,\mathbb C^{d_j})$ such that
$$
I-\Psi(w)^*\Psi(z)=(1-\overline w_1z_1)F_1(w)^*F_1(z)+(1-\overline w_2z_2)F_2(w)^*F_2(z);
$$
\item[(\text{RF})] There exist positive integers $d_1,d_2$, and a unitary
$$
U=\begin{bmatrix} A&B\\C&D \end{bmatrix}=\begin{bmatrix} A&B_1&B_2\\C_1&D_{11}&D_{12}\\C_2&D_{21}&D_{22} \end{bmatrix}:
\begin{bmatrix} \mathbb C^d\\ \mathbb C^{d_1}\\ \mathbb C^{d_2}\end{bmatrix}\to \begin{bmatrix} \mathbb C^d\\ \mathbb C^{d_1}\\ \mathbb C^{d_2}\end{bmatrix}
$$such that
$$
\Psi(z)=A+B\cE_{z}(I-D\cE_z)^{-1}C
$$where $\cE_z=z_1P_1+z_2P_2$ and $P_1,P_2$ are orthogonal projections of $\mathbb C^{d_1+d_2}$ onto $\mathbb C^{d_1}$ and $\mathbb C^{d_2}$, respectively.
\end{enumerate}
Moreover, there is a minimal choice of the integers $d_1,d_2$ in item $(\operatorname{RF})$ above, viz., if $\det \Psi = \tilde g/g$, then $(d_1,d_2)=\deg \tilde g$.
\end{thm}

A subset $\mathbb F$ of $\mathbb D^2$ is said to be {\em symmetric}, if $(z_2,z_1)\in\mathbb F$, whenever $(z_1,z_2)\in\mathbb F$. Given a symmetric subset $\mathbb F$ of $\mathbb D^2$, a function $g$ on $\mathbb F\times\mathbb F$ is said to be {\em doubly symmetric} if it is symmetric in both the coordinates, i.e., denoting $z^\sigma:=(z_2,z_1)$ for $z=(z_1,z_2)$,
$$g(z,w)=g(z^\sigma, w)=g(z, w^\sigma)=g(z^\sigma,w^\sigma)$$ for all $z,w\in\mathbb F$.

The following lemma was proved for the scalar-valued functions in \cite[Lemma 2.3]{AY-R}, and it was observed that the proof works for the operator-valued case as well. We have observed that the proof also works for functions acting on an arbitrary symmetric subset $\mathbb F$ of $\bD^2$ instead of the whole domain. The technique, however, is exactly the same as the one for the scalar case. Therefore we omit the proof.
\begin{lemma}\label{L:DoubSym}
Let $\mathbb F$ be a symmetric subset of $\mathbb D^2$ and $G:\mathbb F\times\mathbb F\to \mathbb C^{d\times d}$ be a doubly symmetric function such that there exist positive integers $d_1,d_2$ and functions $F_j:\mathbb F\to \cB(\mathbb C^d,\mathbb{C}^{d_j})$ such that
\begin{align}\label{G}
G(z, w)= (1-\overline{w_1}z_1)F_1(w)^*F_1(z)+ (1-\overline{w_2}z_2)F_2(w)^*F_2(z)
\end{align}
for all $z,w\in\mathbb F$. Then there exist a unitary operator $\tau$ on $\mathbb{C}^{d_1+d_2}$ and a function $F:\pi(\mathbb{F})\rightarrow\cB(\mathbb C^d,\mathbb{C}^{d_1+d_2})$ such that for all $z,w\in\mathbb F$,
\begin{align}\label{finalG}
G(z, w)=F(t,q)^*(I-\varphi(\tau,t,q)^*\varphi(\tau,s,p))F(s,p),
\end{align}
where $(s,p)=\pi(z)$ and $(t,q)=\pi(w)$.

\end{lemma}

The following theorem shows that the rational inner functions on $\bG$ are exactly the ones that have a finite dimensional realization formula.
\begin{thm}\label{T:SymRatInn}
Let $\Psi$ be a $d\times d$ matrix-valued analytic function on $\mathbb G$. Then the following are equivalent:
\begin{enumerate}
\item[(\text{RI})]$\Psi$ is rational and inner function;
\item[(AD)]There exist positive integers $d_1,d_2$, a unitary operator
\begin{align*}
\tau:\begin{bmatrix}\mathbb C^{d_1}\\ \mathbb C^{d_2}\end{bmatrix}\to\begin{bmatrix}\mathbb C^{d_1}\\ \mathbb C^{d_2}\end{bmatrix}
\end{align*} and an analytic function $F:\mathbb G\to \cB(\mathbb C^d,\mathbb C^{d_1+d_2})$ such that
\begin{align}\label{AD}
I-\Psi(t,q)^*\Psi(s,p)= F(t,q)^*(I-\varphi(\tau,t,q)^*\varphi (\tau,s,p))F(s,p);
\end{align}
\item[(\text{RF})] There exist positive integers $d_1,d_2$ and unitary operators
\begin{align*}
\tau:\begin{bmatrix}\mathbb C^{d_1}\\ \mathbb C^{d_2}\end{bmatrix}\to\begin{bmatrix}\mathbb C^{d_1}\\ \mathbb C^{d_2}\end{bmatrix} \quad \text{and}\quad
U=\begin{bmatrix} A&B\\C&D \end{bmatrix}=\begin{bmatrix} A&B_1&B_2\\C_1&D_{11}&D_{12}\\C_2&D_{21}&D_{22} \end{bmatrix}:
\begin{bmatrix} \mathbb C^d\\ \mathbb C^{d_1}\\ \mathbb C^{d_2}\end{bmatrix}\to \begin{bmatrix} \mathbb C^d\\ \mathbb C^{d_1}\\ \mathbb C^{d_2}\end{bmatrix}
\end{align*}such that
\begin{align}\label{SymPsi}
\Psi(s,p)=A+B\varphi(\tau,s,p)(I-D\varphi(\tau,s,p))^{-1}C.
\end{align}

\end{enumerate}

Moreover, writing $\det \Psi\circ \pi = \tilde g/g$, the integers in items  $(\operatorname{AD})$ and $(\operatorname{RF})$ can be chosen to be $(d_1,d_2)=\deg \tilde g$; and this is the minimal possible integers for such a representation.
\end{thm}
\begin{proof}The strategy is to follow the path (RI)$\Rightarrow$(AD)$\Rightarrow$(RF)$\Rightarrow$(RI).

For (RI)$\Rightarrow$(AD), the idea is to apply Theorem \ref{T:BDisk} to the function $\tilde{\Psi}:\mathbb{D}^2\to\mathbb{C}^{d\times d}$ defined as
$$\tilde{\Psi}(z_1, z_2)= \Psi(z_1+ z_2, z_1z_2) \text { for all } (z_1, z_2)\in{\mathbb{D}^2}.$$
Clearly, $\tilde{\Psi}$ is rational and inner. Therefore by (RI)$\Rightarrow$(AD) of Theorem \ref{T:BDisk}, there exist positive integers $d_1,d_2$ and analytic functions $F_j:\mathbb{D}^2\to \cB(\mathbb C^d,\mathbb{C}^{d_j}), j=1, 2,$ such that
\begin{align*}
I-\tilde{\Psi}(w)^*\tilde{\Psi}(z)=(1-\overline{w_1}z_1)F_1(w)^*F_1(z)+(1-\overline{w_2}z_2)F_2(w)^*F_2(z)
\end{align*}
 for all $z,w\in\mathbb{D}^2$. Since the function $(z,w)\mapsto I-\tilde{\Psi}(w)^*\tilde{\Psi}(z)$ is doubly symmetric on $\mathbb D^2\times\mathbb D^2$, by Lemma \ref{L:DoubSym} there exist a unitary operator $\tau$ on $\mathbb{C}^{d_1+d_2}$ and an analytic map $F:\mathbb{G}\rightarrow\cB(\mathbb{C}^{d_1+d_2},\mathbb C^d)$ such that for all $z,w\in\mathbb{D}^2$,
$$
I-\tilde{\Psi}(w)^*\tilde{\Psi}(z)=F(t,q)^*(I-\varphi(\tau,t,q)^*\varphi(\tau,s,p))F(s,p)=I-{\Psi}(t, q)^*{\Psi}(s, p),
$$
where $(s,p)=\pi(z)$ and $(t,q)=\pi(w)$. This establishes (AD).

(AD)$\Rightarrow$(RF) Rearrange  equation \eqref{AD} to obtain
$$I+F(t, q)^*\varphi(\tau,t,q)^*\varphi (\tau,s,p))F(s,p)=\Psi(t,q)^*\Psi(s,p)+ F(t,q)^*F(s, p).
$$ What follows is a standard lurking isometry argument. Indeed, the above identity implies that there is a unitary 
	$V : \mathcal{D}_{\mathcal{V}} \to \mathcal{R}_{\mathcal{V}}$ where 
\begin{align*}
&	\mathcal{D}_{\mathcal{V}} =\overline{\operatorname{span}} \left\{ \begin{bmatrix}
		I \\
		\varphi(\tau, s, p)F(s, p)
	\end{bmatrix}\eta: (s, p)\in\bG \mbox{ and } \eta\in\mathbb{C}^d\right\} \text{ and}\\
&	\mathcal{R}_{\mathcal{V}}  =\overline{\operatorname{span}} \left\{ \begin{bmatrix}
		\Psi(s,p)\\
	F(s, p)
	\end{bmatrix}\eta: (s, p)\in\bG \mbox{ and } \eta\in\mathbb{C}^d\right\}.
\end{align*}
It could happen that $\mathcal{D}_{\mathcal{V}}$ and/or $\mathcal{R}_{\mathcal{V}}$ is not the whole space. As $V$ is unitary, we know that $\mathcal{D}_{\mathcal{V}}$ and $\mathcal{R}_{\mathcal{V}}$ are of the same dimension. Since the ambient space $\begin{bmatrix}
		\mathbb{C}^d \\ \mathbb{C}^{d_1 + d_2}\end{bmatrix}$ is finite-dimensional, it follows that the original $V : \mathcal{D}_{\mathcal{V}} \to \mathcal{R}_{\mathcal{V}}$ can be extended to a unitary
	\[
	\begin{bmatrix}
		A & B \\
		C & D
	\end{bmatrix} : \mathbb{C}^{d_1 + d_2} \to \mathbb{C}^{d_1 + d_2}.
	\]such that for all $(s,p)\in\mathbb{G}$,
	$$\begin{bmatrix}
		A & B \\
		C & D
	\end{bmatrix}
	\begin{bmatrix}
		I \\
		\varphi(\tau, s, p)F(s, p)
	\end{bmatrix}=\begin{bmatrix}
		\Psi(s, p)\\
		F(s, p)
	\end{bmatrix}.
	$$
Equivalently,
\begin{align*}
A+B\varphi(\tau, s, p)F(s, p)&=\Psi(s, p)\\
C+D\varphi(\tau. s, p)F(s, p)&= F(s, p).
\end{align*}
From these two equations, one easily eliminates $F(s,p)$ to obtain
$$\Psi(s,p)=A+B\varphi(\tau,s,p)(I-D\varphi(\tau,s,p))^{-1}C.$$

 (RF)$\Rightarrow$(RI)
 The rationality of $\Psi$ is easily read off from the expression \eqref{SymPsi}, while the assertion of innerness follows from a straightforward computation that leads to the identity
$$
I-\Psi(s,p)^*\Psi(s,p)=C^*(I-\varphi(\tau,s,p)^*D^*)^{-1}\left(I-\varphi(\tau,s,p)^*\varphi(\tau,s,p)\right)(I-D\varphi(\tau,s,p))^{-1}C
$$for every $(s,p)\in\mathbb G$. Now a straightforward computation shows that $\varphi(\tau,s,p)$ is unitary whenever $(s,p)$ is in $b\bG$ and $\tau$ is unitary. Thus in view of the above equation, the assertion in (RI) follows. The `moreover' part follows from that of Theorem \ref{T:BDisk}.
\end{proof}

Finally, the takeaway of this section is the following result.
\begin{thm}\label{T:RS}
A solvable matrix Nevanlinna--Pick problem on the symmetrized bidisk has a rational inner solution.
\end{thm}
\begin{proof}
Consider an $N$-point solvable matrix Nevanlinna--Pick problem with initial data $(s_1,p_1),(s_2,p_2),\dots,(s_N,p_N)$ in $\mathbb G$ and final data $M_1,M_2,\dots,M_N$ in the closed operator-norm unit ball of $d\times d$ complex matrices. For each $j=1,2,\dots,N$, let $z_j\in\mathbb D^2$ be such that $\pi(z_j)=(s_j,p_j)$. Consider the following (at most $2N$-point) matrix Navanlinna--Pick interpolation problem in $\mathbb D^2$:
 $$
 z_j \mapsto M_j \quad \text{and}\quad z_j^\sigma\mapsto M_j \quad\text{for each }j=1,2,\dots,N.
 $$By hypothesis, this interpolation problem is solvable and thus by \cite[Corollary 2.13]{AM-Bidisc}, there exists a rational inner function $\Psi:\mathbb{D}^2\rightarrow\mathbb{C}^{d\times d}$ such that
\begin{align*}
\Psi(z_j)=M_j=\Psi(z_j^\sigma) \text{ for each }j=1,2,\dots, N.
\end{align*}
Apply Theorem \ref{T:BDisk} to the rational inner function $\Psi$ to get positive integers $d_1,d_2$ and functions $F_{j}:\mathbb{D}^2\rightarrow \cB(\mathbb{C}^{d}, \mathbb{C}^{d_j})$ for $j=1,2$ such that
$$
I-\Psi(w)^*\Psi(z)=(1-\overline w_1z_1)F_1(w)^*F_1(z)+(1-\overline w_2z_2)F_2(w)^*F_2(z).
$$Consider the finite subset $\mathbb F=\{z_j,z_j^\sigma:j=1,2,\dots,N\}$ of $\mathbb D^2$. Define $G:\mathbb F\times\mathbb F\to\mathbb C^{d\times d}$ by
$$
G(z,w)=I-\Psi(w)^*\Psi(z).
$$Since $G$ is doubly symmetric, invoke Lemma \ref{L:DoubSym} to get a unitary $\tau$ on $\mathbb{C}^{d_1+d_2}$ and a function $F:\pi(\mathbb F)\rightarrow\cB(\mathbb{C}^{d},\mathbb{C}^{d_1+d_2})$ such that
$$I-M_i^*M_j=F(s_i,p_i)^*(I-\varphi(\tau,s_i,p_i)^*\varphi(\tau,s_j,p_j))F(s_j,p_j)$$
for each $j=1,2,\dots,N$. Rearrange the above equation to obtain
\begin{align}\label{Condition-Solv}
I+F(s_i,p_i)^*\varphi(\tau, s_i,p_i)^* \varphi(\tau, s_j,p_j) F(s_j,p_j)&=M_i^*M_j+F(s_i,p_i)^*F(s_j,p_j).
\end{align}This readily implies that the operator defined as
\begin{align}\label{Isometry}
\begin{bmatrix}
I_{\mathbb C^d}\\\varphi(\tau, s_j,p_j)F(s_j,p_j)
\end{bmatrix}\xi\mapsto\begin{bmatrix}
M_j\\
F(s_j,p_j)
\end{bmatrix}\xi
\end{align} from
$$\operatorname{span}\left\{
\begin{bmatrix}
I\\\varphi(\tau, s_j,p_j)F(s_j,p_j)
\end{bmatrix}\xi:\xi\in\mathbb C^d,j=1,2,\dots,N\right\}\subset \begin{bmatrix}\mathbb{C}^d\\ \mathbb C^{d_1+d_2}\end{bmatrix}
$$ onto
$$
\operatorname{span}\left\{
\begin{bmatrix}
M_j\\
F(s_j,p_j)
\end{bmatrix}\xi:\xi\in\mathbb C^d,j=1,2,\dots,N\right\}\subset \begin{bmatrix}\mathbb{C}^d\\ \mathbb C^{d_1+d_2}\end{bmatrix}
$$is unitary. By extending this unitary to $ \begin{bmatrix}\mathbb{C}^d\\ \mathbb C^{d_1+d_2}\end{bmatrix}$ and denoting by
 $$
 \begin{bmatrix} A&B\\C&D \end{bmatrix}:\begin{bmatrix}\mathbb C^d\\ \mathbb C^{d_1+d_2}\end{bmatrix}\to \begin{bmatrix}\mathbb C^d\\ \mathbb C^{d_1+d_2}\end{bmatrix}
 $$we define $\Phi:\mathbb{G}\rightarrow\mathbb C^{d\times d}$ by
$$\Phi(s,p)=A+B\varphi(\tau,s,p)(I-D\varphi(\tau,s,p))^{-1}C.$$
By (RF)$\Rightarrow$(RI) of Theorem \ref{T:SymRatInn}, $\Phi$ so defined is rational and inner. Moreover, $\Phi$ interpolates the data because of \eqref{Isometry}.
\end{proof}

\section{Admissible extension of admissible kernels}\label{S:UV}
We begin this section by defining the terminology alluded to in the title above. Here, by a (scaled-valued) {\em kernel} on a set $\Lambda$, we mean a function $k:\Lambda\times \Lambda\to\bC$ such that for every choice of finitely many points $\{\lambda_j\}_{j=1}^N$ in $\Lambda$, the matrix $\left[k(\lambda_i,\lambda_j)\right]_{i,j=1}^N$ is positive semi-definite and the diagonal entries are non-zero. It is well known that associated to every kernel $k$, there is a Hilbert function space $H(k)$, called the reproducing kernel Hilbert space, where the set $\{k(\cdot,\lambda):\lambda\in \Lambda\}$ constitutes a total set for $H(k)$; we refer the readers to \cite[Chapter 2]{AMc-book} for elementary theory of such kernels.
\begin{definition}
Let $\mathcal G$ be a subset of $\mathbb G$. An {\em admissible kernel} $k$ on $\cG$ is a kernel such that the pair $(M_s,M_p)$ of multiplication by coordinate functions on $H(k)$ has $\Gamma$ (the closure of $\bG$) as a spectral set, i.e.,
$$
\|f(M_s,M_p)\|\leq \sup\{|f(s,p)|:(s,p)\in\Gamma\}
$$for every polynomial $f$ in two variables.
\end{definition}
 An example of an admissible kernel on $\bG$ is given by
 $$
((s,p),(t,q))\mapsto \frac{1}{(1-p\overline q)^2-(s-\overline t p)(\overline t- s\overline q)}.
 $$This kernel is known to have all the properties to be referred to as the Szeg\"o kernel for $\bG$ - see the papers \cite{MSRZ, BS, BDSIMRN}. Let us recall that a data set $\{(\lambda_j,w_j):1\leq j\leq N\}$ in $\bD$ is solvable if and only if the Pick matrix
$$
\begin{bmatrix}\frac{1-w_i\overline w_j}{1-\lambda_i\overline\lambda_j}\end{bmatrix}_{i,j=1}^N
$$is positive semi-definite. One way of interpreting this result is that the solvability of a Pick problem in $\bD$ is equivalent to checking the positivity of $1-w_i\overline w_j$ against the Szeg\"o kernel $\bS$ of $\bD$, i.e., $\bS(\lambda_i,\lambda_j):=(1-\lambda_i\overline\lambda_j)^{-1}$. This point of view makes it possible to approach the Pick problem in other cases such as the multi-connected domains \cite{Abrahamse}, the polydisk \cite{Agler1988, AMc-book}, the distinguished varieties \cite{JKM'12} and the symmetrized bidisk. However, one must consider a family of kernels instead of just the Szeg\"o kernel.
\begin{thm}[See \cite{BS}]\label{T:SolveG}
A Pick data set $\{(\lambda_j,w_j):1\leq j\leq N\}$ in $\bG$ is solvable if and only if for every admissible kernel $k$ on $\bG$, the matrix
       \begin{align}\label{PickMatrix}
        \begin{bmatrix}(1-w_i\overline{w_j})k(\lambda_i, \lambda_j)\end{bmatrix}_{i,j=1}^N
       \end{align}is positive semi-definite.
\end{thm}
Let $\Lambda$ be a finite subset of $\bG$ and $k$ on $\Lambda$ be an admissible kernel. As it turns out, there is always a distinguished variety containing $\Lambda$ -- see Remark \ref{R:Substance} below. The question that we address below is whether there is an admissible kernel on the distinguished variety that extends $k$. We also need the extended kernel to possess a type of continuous structure. The concrete dilation theory developed in \cite{BPSR} and the detailed geometry of distinguished varieties developed in \S \ref{S:DistVar} will come in handy now. The result below does the heavy lifting in the proof of Theorem \ref{T:SandIntro}. The proof of part (iii) is due to John E.\ McCarthy from a private communication and is different from the proof of its bidisk analogue (Lemma 4.16 in \cite{AM_Erratum}). We call it a theorem for its significant role in the proof of Theorem \ref{T:SandIntro}. 
\begin{thm}\label{T:AdmisExt}
Let $\Lambda=\{(s_j,p_j):j=1,2,\dots,N\}$ be a subset of $\mathbb G$ and $k$ be an admissible kernel on $\Lambda$. Then there is a distinguished variety $\cW$ containing $\Lambda$ and an admissible kernel $K$ on $\cW\cap\mathbb{G}$ that extends $k$. Moreover, the distinguished variety and the extension can be obtained with the following bonuses:
\begin{enumerate}
\item[(i)] there is a completely non-unitary numerical contraction $F$ such that $\cW=\cW_F$;
\item[(ii)] the extension can be made so that the extended kernel $K$ is given by
$$
K((s,p),(t,q))=\frac{1}{1-p\overline q}\langle u(t,q),u(s,p)\rangle=\langle \bS\otimes u(t,q), \bS\otimes u(s,p)\rangle,
$$where $\bS$ is the Szeg\"o kernel for $\bD$ and $u(s,p)\in\operatorname{Ker} (F+\overline p F^*-\overline sI) $; and
\item[(iii)]the kernel vectors $u(s,p)$ can be chosen so that for each $(s_j,p_j)$, there exist $q_j$ functions $\alpha_1,\alpha_2,\dots,\alpha_{q_j}$ acting on a neighbourhood $V_j$ of $\overline p_j$ so that for each $z\in V_j$ and $l\in\{1,2,\dots,q_j\}$, $\alpha_l(z)$ belongs to $\sigma(F+\overline p_j F^*)$,
    $$
     \alpha_l(z) \to \overline s_j
  \quad \text{and}\quad
    \sum_{l=1}^{q_j}u(\overline{\alpha_l(z)},\overline z)\to u(s_j,p_j) \quad\text{as}\quad z\to\overline{p_j}.
    $$

\end{enumerate}
\end{thm}
\begin{proof}
Let $k$ be an admissible kernel on the set $\Lambda$, i.e., the pair
 $(M_s,M_p)$ of multiplication by coordinate functions on $H(k)=\operatorname{span}\{k_{(s_j,p_j)}: j=1,2,...,N\}$ has $\Gamma$ as a spectral set (i.e., a $\Gamma$-contraction). Let $F'$ on $\cF'=\operatorname{Range}(I-M_pM_p^*)^{1/2}$ be the numerical contraction given by \cite[Theorem 4.2]{BPSR} applied to $(M_s^*,M_p^*)$, i.e., $F'$ is the unique numerical contraction on $\cF'$ (the so-called fundamental operator) such that
 \begin{align}\label{FundEqn}
 M_s^*-M_sM_p^*=D_{M_p^*}F'D_{M_p^*},
 \end{align}where $D_{M_p^*}=(I-M_pM_p^*)^{1/2}$. Since $(s_j,p_j)\in\bG$, we have $|p_j|<1$ for each $j$ and therefore $M_p^{*n}\to 0$ strongly as $n\to\infty$. Using this convergence, one can compute that the operator $\Pi:H(k)\rightarrow H^2(\cF')$ given by
 \begin{align}\label{Dil_1}
 \Pi h= \sum_{n\geq{0}}z^n D_{M_p^*} M_p^{*n} h \quad\text{for all}\quad h\in H(k)
\end{align}is an isometry. The following intertwining relations were established in the proof of \cite[Theorem 4.6]{BPSR}:
\begin{align}\label{Dil_2}
\Pi(M_s^*, M_p^*)= (M_{F'^*+zF'}^*, M_z^*)\Pi.
\end{align}Using the first intertwining relation, we compute
\begin{align*}
\overline{s_j}\Pi k_{(s_j,p_j)}=\Pi M_s^*k_{(s_j,p_j)}&=M_{F'^*+zF'}^*\Pi k_{(s_j,p_j)}=M_{F'^*+zF'}^*\sum_{n\geq0}z^n\overline{p_j}^nD_{M_p^*}k_{(s_j,p_j)}\\
&=\sum_{n\geq{0}}z^n\overline{p_j}^n F'D_{M_p^*} k_{(s_j,p_j)}+\sum_{n\geq{0}}z^n\overline{p_j}^{n+1} F'^*D_{M_p^*}k_{(s_j,p_j)}\\
&=\sum_{n\geq{0}}z^n\overline{p_j}^n (F'+\overline{p_j}F'^*)D_{M_p^*}k_{(s_j,p_j)}.
\end{align*}Equating the coefficients we thus obtain
\begin{align*}
(F'+\overline{p_j}F'^*-\overline{s_j}I)D_{M_p^*}k_{(s_j,p_j)}=0 \quad \text{for each}\quad j=1,2,\dots,N.
\end{align*}Let $F'=\sbm{F&0\\0&F_u}:\sbm{\cF\\\cF_u}\to\sbm{\cF\\\cF_u}$ be the decomposition of $F'$ into the c.n.u.\! part $F$ and the unitary part $F_u$. We pause here to note that none of the vectors $D_{M_p^*}k_{(s_j,p_j)}$ has a non-zero component in $\cF_u$. Indeed, if there was a non-zero vector $v$ such that
$(F_u+\overline{p_j}F_u^*-\overline{s_j}I)v=0$ for some $j$, then we would have $s_j=\beta+\overline{\beta}p_j$ for some $\beta\in\bT$. This contradicts part (iii) of Theorem \ref{T:Charc_G} because $(s_j,p_j)$ comes from $\bG$. Therefore we must have
\begin{align}\label{Points_condi}
(F+\overline{p_j}F^*-\overline{s_j}I)D_{M_p^*}k_{(s_j,p_j)}=0 \quad \text{for each}\quad j=1,2,\dots,N,
\end{align}which readily implies that the points $(s_j,p_j)$ belong to
\begin{align}\label{DV}
\cW_{F}=\{(s, p)\in\mathbb C^2: \operatorname{det}(F^*+pF-sI)=0\},
\end{align} which is a distinguished variety by Theorem \ref{T:DistMain}.

To produce an admissible extension of $k$ to $\cW_F\cap\mathbb G$, we pick an $(s,p)$ in $(\cW_F\cap\mathbb G)\setminus\Lambda$ and choose $u(s,p)$ to be {\em any} vector in $\operatorname{Ker}(F+\overline pF^*-\overline s I)$ and set $u(s_j,p_j)=D_{M_p^*}k_{(s_j,p_j)}$ for each $j=1,2,\dots,N$. We shall prove that the function $K:(\cW_F\cap\mathbb G)\times (\cW_F\cap\mathbb G)\to\bC$ given by
$$
K((s,p),(t,q))=\langle \bS_{ q}\otimes u(t,q), \bS_{p}\otimes u(s,p)\rangle_{H^2\otimes \cF}
$$is an admissible kernel that extends $k$. That it extends $k$ follows from the following easy computation:
 \begin{align*}
K((s_i,p_i),(s_j,p_j))&=\frac{1}{(1-p_i\overline{p_j})}\langle D_{M_p^*}k_{(s_j,p_j)}, D_{M_p^*}k_{(s_i,p_i)}\rangle\\
 & =\frac{1}{(1-p_i\overline{p_j})}\langle(I-M_pM_p^*)k_{(s_j,p_j)},k_{(s_i,p_i)}\rangle=k((s_i,p_i),(s_j,p_j)).
 \end{align*} To prove the admissibility of $K$, we show that the pair of multiplication by coordinate functions $(M_s,M_p)$ on $H(K)=\overline{\operatorname{span}}\{K_{(s,p)}:(s,p)\in\cW_F\cap\mathbb G\}$ is a $\Gamma$-contraction. To that end, we first observe that $H(K)$ is unitarily equivalent to
 $$
\cR=\overline{\operatorname{span}}\{\bS_p\otimes u(s,p):(s,p)\in\cW_F\cap \mathbb G\}\subseteq H^2\otimes\cF
 $$via the unitary given by
 $$
 \tau:\sum_{l=1}^dc_lK_{(s_l,p_l)}\mapsto \sum_{l=1}^dc_l\bS_{p_l}\otimes u(s_l,p_l) \quad \text{where} \quad (s_l,p_l)\in\cW_F\cap \mathbb G .
 $$ We show that $(M_s^*,M_p^*)$ on $H(K)$ is unitarily equivalent to $(M_{F^*+zF}^*,M_z^*)|_\cR$ via the unitary $\tau$. This is easily achieved by doing the following computation on the kernel functions:
\begin{align*}
M_{F^*+zF}^* \tau K_{(s,p)}=M_{F^*+zF}^*(\bS_{ p}\otimes u(s,p))&= \bS_{ p}\otimes Fu(s,p)+\overline{p}\bS_p\otimes F^*u(s,p)\\
&=\bS_{ p}\otimes(F+\overline{p}F^*)u(s,p)\\
&=\overline{s}\bS_{ p}\otimes u(s,p)=\tau M_s^*K_{(s,p)}.
\end{align*}
Note that pairs of the form $(M_{F^*+zF}^*,M_z^*)$ are known to be $\Gamma$-contractions whenever $F$ is a numerical contraction; see \cite[Theorem 2.4]{ay-jot}. Since restrictions and the adjoint of a $\Gamma$-contraction are again  $\Gamma$-contractions, the admissibility of $K$ is established. This completes the proof of items (i) and (ii).

Proof of item (iii) requires some ideas in perturbation theory for the eigenvalue problem: a study of how the eigenvalues and the eigenvectors of a matrix-valued analytic function change subject to a small perturbation. The results that we shall use can be found in the books \cite{Knopp, Kato}. Although we shall apply the results to a very simple matrix-valued function, viz., $z\mapsto F+zF^*$, we briefly discuss the general theory below.

Let $\Psi$ be an $N\times N$ matrix-valued polynomial (the theory extends to the analytic case as well but we restrict ourselves to just polynomials). It follows from the discussion in \cite[Sec.\ 13, Ch.\ 5]{Knopp} that the function
$$
z\mapsto \# \sigma(\Psi(z))
$$ is constant, say $r\leq N$, except at only a finite number of points - the so-called {\em exceptional points}, which may of course be empty. The eigenvalues of $\Psi(z)$ - the so called {\em algebraic functions} - are the roots of the algebraic equation
$$
\det(\Psi(z)-wI)=(-1)^Nw^N+g_1(z)w^{N-1}+\cdots+g_{n-1}(z)w+g_N(z)=0.
$$The algebraic functions are (branches of) analytic functions with exceptional points as the possible branch points. To demonstrate, consider the two examples
$$\begin{bmatrix}
1&z\\z&-1	
\end{bmatrix}
\quad\text{and}\quad
\begin{bmatrix}
	0&z\\z&0
\end{bmatrix}.$$
For the first example, the eigenvalues are $\pm (1+z^2)^{1/2}$, with the exceptional points $\pm i$ as the branch points; for the second example, the
eigenvalues are the entire functions $\pm z$, and $0$ is the only exceptional point.

For a non-exceptional point $z$, let $\alpha_1(z),\alpha_2(z),\dots,\alpha_r(z)$ be the distinct eigenvalues of $\Psi(z)$. Let $P_j(z)$ be the projection (i.e., idempotent but not necessarily self-adjoint) onto the generalized eigenspace corresponding to the eigenvalue $\alpha_j(z)$. The projections $P_j(z)$ are called the {\em eigenprojections}. It is known (see \cite[Sec. I-3,4 and Sec. II-4]{Kato}) that
$$
P_i(z)P_j(z)=P_j(z)P_i(z)=\delta_{ij}P_i(z) \quad\text{and}\quad\sum_{j=1}^rP_j(z)=I_N.
$$Therefore, with the {\em eigennilpotent} defined as
$$
D_j(z)=(\Psi(z)-\alpha_j(z)I)P_j(z)
$$it follows that
$$
\Psi(z)=\sum_{j=1}^r\alpha_j(z)P_j(z)+\sum_{j=1}^rD_j(z).
$$The above decomposition is referred to as the canonical decomposition of $\Psi$; see the discussion in \cite[Sec. I-4 and II-4]{Kato}.

Let $z_0$ be any complex number (possibly an exceptional point) and $w_0$ be an eigenvalue of $\Psi(z_0)$. Let $l$ be the algebraic multiplicity of $w_0$. Let $\epsilon>0$ be small enough so that $\bD(w_0,\epsilon)$, the disk of radius $\epsilon$ around $w_0$, does not contain any other eigenvalue of $\Psi(z_0)$. Then by the ``Theorem on the continuity of the roots" in \cite[Page 122]{Knopp}, it is possible to find a $\delta=\delta(\epsilon)>0$ so that for every $z$ in $\mathbb D(z_0,\delta)$, the disk of radius $\delta$ around $z_0$, there are {\em exactly} $l$ eigenvalues of $\Psi(z)$, say $\alpha_1(z),\alpha_2(z),\dots,\alpha_l(z)$ lying inside $\mathbb D(w_0,\epsilon)$. In other words, every $l$-fold eigenvalue of $\Psi(z_0)$ branches off into exactly $l$ simple roots $\alpha_1(z),\alpha_2(z),\dots,\alpha_l(z)$ such that for each $j$, $\alpha_j(z)\to w_0$ as $z\to z_0$ (compare it with the examples above).  Define the operator
\begin{align}\label{TheP}
P(z)=\frac{1}{2\pi i}\int_{\mathbb D(w_0,\epsilon)}(\zeta-\Psi(z))^{-1}d\zeta.
\end{align}Then the crux of the matter is that $z\mapsto P(z)$ is projection-valued and holomorphic near $z_0$. Furthermore, it is equal to the sum of the eigenprojections corresponding to all the eigenvalues that are inside $\mathbb D(w_0,\epsilon)$, i.e.,
\begin{align}\label{PropertyP}
P(z)=\sum_{j=1}^lP_j(z).
\end{align}It should, however, be noted that while $P$ is holomorphic near $z_0$, the eigenprojections $P_j$ may not even be defined at the exceptional points; see for example \cite[Example II-1.12]{Kato}. For the results stated above, see \cite[Prob. I-5.9 and Sec. II-4]{Kato}.

Let us now consider the example $\Psi(z)=F+zF^*$ for some matrix $F$. It can be checked that $\Psi(z)$ is a normal matrix for each $z$ on the unit circle. Therefore by Theorem II-1.10 in Kato \cite{Kato}, $D_j(z)=0$ for every $z$ except possibly the exceptional ones. In other words, $\Psi(z)$ is diagonalizable at every $z$ except possibly the exceptional points, i.e., for each $j$, $\operatorname{Range}P_j(z)$ equals to the eigenspace corresponding to the $j$-th eigenvalue.

Let us recall that the vectors $u(s,p)$ in item (ii) were just any vectors in the kernel of $F+\overline p F^*-\overline sI$. To prove item (iii), we have to make certain choices of such vectors. To that end, let us pick the point $(s_1,p_1)$; the analysis for the other nodes remains the same. Note that $\overline s_1$ is an eigenvalue of $\Psi(\overline p_1)$ with $u(s_1,p_1)=D_{M_p^*}k_{(s_1,p_1)}$ as an eigenvector. We apply the above analysis for the choice $(z_0,w_0)=(\overline p_1,\overline s_1)$ to obtain, for a sufficiently small $\epsilon>0$, a $\delta>0$ such that for every $z$ in $\mathbb D(\overline p_1,\delta)\cap\mathbb D$, there exist points $\alpha_1(z),\alpha_2(z),\dots,\alpha_l(z)$ in $\mathbb D(\overline s_1,\epsilon)\cap\sigma(\Psi(z))$, i.e., for each $j=1,2,\dots,l$,
\begin{align}\label{convergence}
\alpha_j(z)\to \overline s_1 \quad\text{as}\quad z\to \overline p_1.
\end{align} With the projection $P$ as defined in \eqref{TheP}, we then have $P(\overline p_1)u(s_1,p_1)=u(s_1,p_1)$. For each $z$ in $\mathbb D(\overline p_1,\delta)\cap\mathbb D$, set
$$
v(z)=P(z)u(s_1,p_1)=\sum_{j=1}^l P_j(z)u(s_1,p_1)=\sum_{j=1}^l v_j(z),
$$ where $v_j(z):=P_j(z)u(s_1,p_1)$. Now if necessary we can choose smaller $\delta$ so that $\mathbb D(\overline p_1,\delta)\cap\mathbb D$ does not contain any exceptional point of $\Psi$ (except possibly $\overline p_1$ itself).  Therefore we can assume that $\Psi(z)$ is diagonalizable for each $z\in \mathbb D(\overline p_1,\delta)\cap\mathbb D$. Therefore the vectors $v_j(z)$ are just the eigenvectors corresponding to $\alpha_j(z)$, i.e.,
\begin{align}\label{VarietyMember}
(\Psi(z)-\alpha_j(z)I)v_j(z)=(F+zF^*-\alpha_j(z)I)v_j(z)=0
\end{align}for every $j=1,2,\dots,l$ and each $z\in \mathbb D(\overline p_1,\delta)\cap\mathbb D$. This together with \eqref{convergence} imply that the points $(\overline{\alpha_j(z)},\overline z)$ are in $\cW_F\cap \mathbb G$ and converge to $(s_1,p_1)$ as $z\to\overline p_1$. The proof of item (iii) now follows if we set $u (\overline{\alpha_j(z)},\overline z)=v_j(z)$.
\end{proof}
\begin{remark}\label{R:Substance}
  The substance of Theorem \ref{T:AdmisExt} above is not that there is a distinguished variety containing the finite set $\Lambda$; it is that there is a distinguished variety $\cW$ to which a given admissible kernel on $\Lambda$ can be admissibly extended to $\cW\cap \mathbb G$. Indeed, given a finite subset $\{(s_j,p_j):1\leq j\leq N\}$ of $\bG$, there is a quick way to get a distinguished variety containing these points. For example, consider the diagonal matrix $F=\operatorname{diag}(\beta_j)$. By part (iii) of Theorem \ref{T:Charc_G}, this is a strict contraction and hence a strict numerical contraction. One can show using the relation between $(s_j,p_j)$ and $\beta_j$ given in part (iii) of Theorem \ref{T:Charc_G} that the distinguished variety $\cW_F$ contains all the points $(s_j,p_j)$.
\end{remark}

\section{Distinguished varieties and extremal Pick problems -- proofs of Theorem \ref{T:SandIntro}} \label{S:TwoProofs}
In this section, we give two proofs of Theorem \ref{T:SandIntro}, which we restate below as Theorem \ref{T:Sand} for readers' convenience. We first present the shorter proof that directly appeals to its bidisk version. 
We begin with a couple of elementary observations about the uniqueness set of a solvable problem.
\begin{obs}\label{O:UV}
The uniqueness set corresponding to a solvable data in $\bG$ coincides with an algebraic variety.
\end{obs}
\begin{proof}
Let $\fD=\{((s_j,p_j),w_j):1\leq j\leq N\}$ be a solvable data set in $\bG$ and $\cS$ be the uniqueness set. Let $(s,p)$ be any point in $\bG\setminus\cS$. Then by definition of the uniqueness set, there exists a pair of interpolants $f$ and $f'$ so that $w=f(s,p)\neq f'(s,p)=w'$. Since the data $\fD\cup\{((s,p),w)\}$ and $\fD\cup\{((s,p),w')\}$ are solvable, by Theorem \ref{T:RS}, there exist rational solutions to both of these problems. Consequently, $\cS$ is contained in the common zero sets of the differences of these rational interpolants, which is an algebraic variety; denote this algebraic variety by $\cW_{(s,p)}$. Repeat the analysis for all $(s,p)\in\bG\setminus\cS$ and denote the intersection of $\cW_{(s,p)}$ by $\cW$. Obviously, $\cW$ is an algebraic variety that coincides with $\cS$ in $\bG$, i.e., $\cS=\bG\cap\cW.$
\end{proof}
\begin{obs}\label{O:NonExt}
The uniqueness set for a non-extremal data $\fD=\{((s_j,p_j),w_j):1\leq j\leq N\}$ in $\bG$ is just the set $\{(s_j,p_j):1\leq j\leq N\}$.
\end{obs}
\begin{proof} Let $f$ be a solution of supremum norm less than $1$ and $\eta\neq 0$ be any complex number. Choose $\epsilon>0$ small enough so that for every $0<\delta\leq\epsilon$
the function
$$
h_{\delta}(s,p)= f(s,p)+\delta\prod_{r=1}^N[(s-s_r)+\eta(p-p_r)],
$$ which obviously interpolates, has supremum norm no greater than one. Now for $(s,p)$ to be in the uniqueness set we must have $\prod_{r=1}^N[(s-s_r)+\eta(p-p_r)]=0$, which implies that there is an $r=1,2,\dots,N$, so that $(s-s_r)+\eta(p-p_r)=0$. Note that if $p=p_r$, then obviously $s=s_r$. If $p\neq p_r$, then we must have $\eta=-(s-s_r)/(p-p_r)$. So we choose $\eta$ other than these points. Hence the observation is made.
\end{proof}

An $N$-point extremal Pick problem in $\bG$ is said to be {\em minimal}, if none of the $(N-1)$-point subproblems is extremal.
\begin{thm}\label{T:Sand}
Given a minimal extremal Pick problem $\{(s_j,p_j)\to w_j\}_{j=1}^N$ on the symmetrized bidisk with $\cU$ as its uniqueness variety, there is a distinguished variety $\cW$ such that
\begin{align}\label{Sand}
\{(s_1,p_1),(s_2,p_2),\dots,(s_N,p_N)\}\subset\cW\subseteq\cU.
\end{align}
\end{thm}
\subsection{The first proof}
Our point of departure for the first proof is the following lemma. As in Section \ref{S:RationalFunctions}, we continue to denote $z^\sigma=(z_2,z_1)$ for $z=(z_1,z_2)\in\bC^2$.
\begin{lemma}\label{L:Added}
Let $\{(\lambda_j,w_j):1\leq j\leq N\}$ be an extremal Pick problem in $\bG$. Let $\mu_j\in\bD^2$ be such that $\pi(\mu_j)=\lambda_j$ for each $j=1,2,\dots,N$. Then the (at most $2N$-point) Pick problem $\{(\mu_j,w_j),(\mu_j^\sigma,w_j):1\leq j\leq N\}$ in $\bD^2$ is extremal.
\end{lemma}
\begin{proof}
The proof is rather straightforward. Let $f:\bD^2\to\bD$ solve the Pick data in $\bD^2$. We want to show that $\sup_{\bD^2}|f|=1$. Let us denote the symmetric part of $f$ by $f_+$, i.e., $f_{+}(z)=(f(z)+ f(z^\sigma))/2.$ Since $f(\mu_j)=w_j=f(\mu_j^\sigma)$, we must have $f_+(\mu_j)=f_+(\mu_j^\sigma)=w_j$ for each $j=1,2,\dots,N$. Then the function $F:\bG\to\bD$ defined as $F\circ\pi=f_+$ interpolates the data $\{(\lambda_j,w_j):1\leq j\leq N\}$, and thus by hypothesis, $\sup_{\bG}|F|=\sup_{\bD^2}|f_+|=1$. It now follows from the trivial inequality $|f_+|\leq |f|$ that $\sup_{\bD^2}|f|=1$, and hence the pulled back data in $\bD^2$ must be extremal.
\end{proof}

\noindent
{\em First proof of Theorem \ref{T:Sand}:}
Let $\{(\lambda_j,w_j):1\leq j\leq N\}$ be an extremal and minimal Pick problem in $\bG$, and $\cU$ be its uniqueness variety. We have to find a distinguished variety satisfying the conclusion of Theorem \ref{T:Sand}. Consider the pulled back problem $\{(\mu_j,w_j),(\mu_j^\sigma,w_j):1\leq j\leq N\}$ in $\bD^2$. This problem is extremal by Lemma \ref{L:Added}. If it is not minimal, then by definition it has an extremal sub-problem. Note that the initial data of this sub-problem cannot miss out both $\mu_j$ and $\mu_j^\sigma$ for any $j=1,2,\dots,N$. This is because of the trivial fact that if a Pick data $\{(\mu_j,w_j):1\leq j\leq N\}$ in $\bD^2$ is extremal, then so is the problem $\{(\pi(\mu_j),w_j):1\leq j\leq N\}$ in $\bG$. Therefore if the extremal sub-problem misses both $\mu_j$ and $\mu_j^\sigma$ for some $j$, then we shall have an extremal sub-problem of $\{(\lambda_j,w_j):1\leq j\leq N\}$; this contradicts the minimality of the problem. Consequently, if the Pick problem $\{(\mu_j,w_j),(\mu_j^\sigma,w_j):1\leq j\leq N\}$ in $\bD^2$ is not extremal, then we shall end up having a minimal extremal sub-problem $\fD=\{(\delta_j,w_j): j=1,2,\dots, R\}$, where $R\leq 2N$ such that $\{\pi(\delta_j):j=1,2,\dots,R\}=\{\lambda_j:j=1,2,\dots,N\}$. This is good enough for us. Because then by Theorem 4.1 of \cite{AM_Acta} (the bidisk analogue of Theorem \ref{T:Sand}), there exists a distinguished variety $\cV$ with respect to $\bD^2$ such that $\{\delta_j\}_{j=1}^R\subset\cV\subset\cU'$, where $\cU'$ is the uniqueness variety for $\fD$. Note next that if $F$ is a solution for $\{(\lambda_j,w_j):1\leq j\leq N\}$, then $F\circ\pi$ is a solution for $\fD$, and therefore, $\pi(\cU')\subset\cU$. The proof now follows once we take the image of the containments $\{\delta_j\}_{j=1}^R\subset\cV\subset\cU'$ under $\pi$.
\qed

\subsection{The second proof}
From a practical point of view, it is in general inconvenient to approach a given Pick problem in $\bG$ by pulling it back to $\bD^2$. Because as the analysis in the first proof of Theorem \ref{T:Sand} above shows, there is a two-fold difficulty with this pull-back approach: (i) One may have to consider a Pick problem with an increased number of Pick data when pulled back to $\bD^2$. The difficulty of a Pick problem obviously increases with the number of Pick data; and (ii) It may be difficult, in general, to find out which sub-problem of the pulled back problem is minimal and extremal. Therefore it is desirable to have a direct proof staying in the realm of the symmetrized bidisk.

Given a solvable data $\{(\lambda_j,w_j):1\leq j\leq N\}$ in $\bG$, an admissible kernel $k$ is called {\em active}, if the (positive semi-definite) matrix $\begin{bmatrix}
(1-w_i\overline w_j)k(\lambda_i,\lambda_j)
\end{bmatrix}$ has a non-zero vector in its null space. There is an easy way (if one is lucky) to tell if a given solvable data is extremal.
\begin{lemma}\label{L:Active}
A solvable Pick problem on the symmetrized bidisk is extremal if and only if it has an active kernel.
\end{lemma}
The proof of this lemma goes along the same lines as that of Lemma 4.4 in \cite{AM_Acta}; thus we omit it. In the following, we shall use the following convention of notation
$$k((s_i,p_i),(s_j,p_j))=:k_{ij}.$$

\noindent
{\em Second proof of Theorem \ref{T:Sand}:}
Since the Pick problem is extremal, by Lemma \ref{L:Active}, there is an active kernel $k$, i.e., $k$ is an admissible kernel such that the corresponding Pick matrix has a non-trivial kernel. Let $\gamma=(\gamma_1,\gamma_2,\dots,\gamma_N)^T$ be a non-zero vector such that
\begin{align}\label{gamma}
\big[(1-w_i\overline{w_j})k_{ij}\big]\gamma=0.
\end{align}Note that no component of $\gamma$ can be zero, for otherwise, an $(N-1)$-point sub-problem would have an active kernel, and by Lemma \ref{L:Active}, this $(N-1)$-point Pick problem would be extremal, which contradicts the minimality of the problem.

Invoke Theorem \ref{T:AdmisExt} to obtain a distinguished variety $\cW$ containing each of $(s_j,p_j)$ and an admissible kernel $K$ on $\cW\cap \mathbb G$ extending $k$ with additional properties (i), (ii) and (iii). We shall show that possibly a subvariety of this distinguished variety satisfies the containments in \eqref{Sand}. We shall establish the second containment only in the domain $\bG$. Note that this suffices because every irreducible component of $\cW$ intersects $\bG$. The strategy is to show that if $(s_{N+1}, p_{N+1})$ is any point in $\cW\cap \mathbb G$ other than the nodes $(s_j,p_j)$ and $w_{N+1}$ is complex number assumed by some interpolant at $(s_{N+1}, p_{N+1})$, then $w_{N+1}$ must not depend on the interpolant.

This will take a fair bit of analysis. We start with the observation that since the Pick problem $\{(s_j,p_j)\to w_j\}_{j=1}^{N+1}$ is solvable (by the choice of the point $w_{N+1}$) and $K$ is admissible, by Theorem \ref{T:SolveG}, we have
$$
\big[(1-w_i\overline{w_j})K_{ij}\big]_{i,j=1}^{N+1}\succeq{0},
$$which in particular implies that
\begin{align}\label{N+1P}
\big\langle\big[(1-w_i\overline{w_j})K_{ij}\big]\begin{pmatrix}
\gamma \\
\delta
\end{pmatrix},
\begin{pmatrix}
\gamma \\
\delta
\end{pmatrix}\big\rangle\geq{0},
\end{align}
where $\gamma$ is as in \eqref{gamma} and $\delta$ is any complex number. Unfolding the LHS of \eqref{N+1P} and making use of \eqref{gamma} we get
\begin{align}\label{L1}
2\operatorname{Re}\big[\overline{\delta}\sum_{j=1}^N(1-w_{N+1}\overline{w_j})K_{N+1,j}\gamma_{j}\big]+|\delta|^2(1-|w_{N+1}|^2)K_{N+1,N+1}\geq{0}.
\end{align}
Since the above inequality holds for all $\delta\in\mathbb{C}$, it follows from a routine analysis that
$$\sum_{j=1}^N(1-w_{N+1}\overline{w_j})K_{N+1,j}\gamma_{j}=0,$$
which, after a rearrangement of terms, leads to
\begin{align}\label{L2}
w_{N+1}\big(\sum_{j=1}^N\overline{w_j}K_{N+1, j} \gamma_j\big)=\sum_{j=1}^NK_{N+1,j}\gamma_j.
\end{align}
This clearly yields an expression of $w_{N+1}$ depending only on the interpolation data and the point $(s_{N+1},p_{N+1})$, viz.,
\begin{align}\label{L3-Zero}
w_{N+1}=\frac{\sum_{j=1}^NK_{N+1, j} \gamma_j}{\sum_{j=1}^N\overline{w_j}K_{N+1,j}\gamma_j},
\end{align}
provided that the denominator does not vanish identically.

In the remainder of the proof, we make sure that at each node $(s_j,p_j)$, there is at least one sheet of $\cW$ passing through $(s_j,p_j)$ that contains a sequence converging to $(s_j, p_j)$ at each term of which the denominator of \eqref{L3-Zero} is non-zero. Then every interpolant agrees along the sequence and thus by the Identity Theorem, every interpolant coincides throughout the sheet. We shall then take the union of those irreducible components of $\cW$ that contains these sheets and discard the other components of $\cW$, if there are any. Note that the resulting variety will possibly be a sub-variety of $\cW$ which would nevertheless satisfy \eqref{Sand}.

We consider the node $(s_1,p_1)$; the analysis for the other nodes remains the same. Item (iii) of Theorem \ref{T:AdmisExt} will come in handy now. By part (iii) of Theorem \ref{T:AdmisExt}, there are $q_1$ sequences $\{(s_{n,i},p_n)\}_{n=1}^\infty$, $i=1,2,...,q_1$ in $\cW\cap\mathbb G$, each of which converges to $(s_1,p_1)$ and
\begin{align}\label{Finally}
\sum_{i=1}^{q_1}u(s_{n,i},p_n) \rightarrow u(s_1,p_1)\quad\text{as}\quad n\to\infty.
\end{align}
Suppose on the contrary that there is no sheet passing through $(s_1,p_1)$ which contains a sequence converging to $(s_1,p_1)$ and the denominator of \eqref{L3-Zero} vanishes at each of the points in the sequence. Then there is a neighbourhood of $(s_1,p_1)$ in $\mathbb G$ such that the denominator vanishes identically in the intersection of the neighbourhood with every sheet passing through $(s_1,p_1)$. This means that the denominator would vanish for the choice $(s_{N+1},p_{N+1})=(s_{n,i},p_n)$ for every $n$ and $i$, i.e.,
\begin{align*}
0=\sum_{j=1}^N\overline{w_j}K\big((s_j, p_j), (s_{n,i},p_n)\big)\gamma_j&=\big\langle\sum_{j=1}^N\overline{w_j}\gamma_j\langle \mathbb{S}_{\overline{p_j}}, \mathbb{S}_{\overline{p_n}}\rangle u(s_j, p_j), u(s_{n,i}, p_n)\big\rangle.
\end{align*}
Since this is true for every $i=1,2,\dots,q_1$ and $n\in{\mathbb{N}}$, in view of the convergence \eqref{Finally}, we therefore have
\begin{align}\label{L4}
\sum_{j=1}^N\overline{w_j}K_{1,j}\gamma_j=0.
\end{align}
The equality above jeopardizes the minimality of the extremal problem $\{\lambda_j\to w_j\}_{j=1}^N$. Indeed, if $\epsilon$ is any non-zero number, then the interpolation problem $\{\lambda_j\to w_j'\}_{j=1}^N$ where $w_1'=w_1+\epsilon$ and $w_j'=w_j$ for $j\geq 2$ is not solvable. This is easily seen because, with the admissible kernel $K$ and the non-zero vector $\gamma$ as in \eqref{gamma}, we have
\begin{align*}
\big\langle\big[((1-w_i'\overline{w'_j})K_{ij})\big]\gamma, \gamma\big\rangle&=
\sum_{i,j=1}^N(1-w'_i\overline{w'_j})K_{ij}\gamma_j\overline{\gamma_i}\\
&=\sum_{i,j=1}^N(1-w_i\overline{w_j})K_{ij}\gamma_j\overline{\gamma_i}-2\operatorname{Re}\big[\epsilon\overline{\gamma_1}\sum_{j=1}^N\overline{w_j}K_{1,j}\gamma_j\big]-|\epsilon|^2K_{11}|\gamma_1|^2\\
&=-|\epsilon|^2K_{11}|\gamma_1|^2<0. \quad\text{[by  \ref{gamma} and \ref{L4}]}
\end{align*}
Here we are using the fact that $K$ does not vanish at the diagonal entries and that none of the entries of $\gamma$ is zero. From the fact that the problem $\{\lambda_j\to w_j'\}_{j=1}^N$ is not solvable for any $\epsilon\neq 0$, we conclude that the value $w_1$ at the node $\lambda_1$ is uniquely determined by the other values at $\lambda_2, \lambda_3,\dots,\lambda_{N}$. This makes the $(N-1)$-point problem $\{\lambda_j\to w_j\}_{j=2}^N$ extremal, contradicting the minimality of the problem.

Therefore, equation \eqref{L3-Zero} gives a unique representation formula for a function interpolating the data along some sheets. This was to be proved.
\qed

\begin{remark}
Given a distinguished variety $\cW$, it is possible to construct a suitable extremal interpolation data for which the set $\cW\cap\bG$ is the uniqueness set. This amounts to a converse of Theorem \ref{T:Sand}, which was recently proved by the authors in \cite[see Theorem 2.10 and Remark 2.11]{DKS}.
\end{remark}

\section{Examples}
In this section, we consider several viable examples of Pick problems and compute their uniqueness varieties explicitly.  First, we consider an example of a $2$-point problem with a unique solution and hence is extremal.
\begin{example}\label{this}
Consider the data $\{((-i/{\sqrt{2}}, 0), 0),((1/{\sqrt{2}}, 0), 1/{\sqrt{2}})\}$ in $\mathbb{G}$. Define the holomorphic map $f:\mathbb{D}\rightarrow \mathbb{G}$ by setting $\lambda\rightarrow(\lambda+ a(\lambda), \lambda\cdot a(\lambda))$, where $a$ is the automorphism of $\mathbb{D}$ given by 
$a(\lambda)=i{(\sqrt{2}\lambda-1)}{(\sqrt{2}-\lambda)}^{-1}$. It is easy to see that $f(0)=(-i/{\sqrt{2}}, 0)$ and $f(1/{\sqrt{2}})=(1/{\sqrt{2}}, 0)$. Therefore by Schwarz lemma any $F:\mathbb{G}\rightarrow{\mathbb{D}}$ interpolating the data is a left inverse of $f$, and vice versa. On the other hand, since $a(\lambda)=\lambda$ has a double root on $\mathbb{T}$ viz.,  $h=(1-i)/{\sqrt{2}}$, by Theorem 5.4 of \cite{KZ_JGA}, $f$ has a unique left inverse. Moreover, the left inverse is given in Lemma 5.7 of \cite{KZ_JGA}, which in this case simplifies to the formula 
$$F(s,p)=\frac{\sqrt{2}(i+\sqrt{2}s-(1+2i)p)}{(3i+1)+\sqrt{2}(1-i)s-(1-i)p}.$$
Since $f$ has a unique left inverse, the solution is unique and hence the uniqueness set is the whole of $\mathbb{G}$.
\end{example}
Now we consider two examples of $2$-point extremal problems such that the solution is not unique. The non-uniqueness of the solutions can also be explained by Theorem 5.4 of \cite{KZ_JGA}.
\begin{example}\label{that}
  Consider the data $\{((0,0),0), ((1,1/4),1/2)\}$ in $\bG$. While both the functions $\frac{s}{2}$ and $-\frac{2p-s}{2-s}$ solve the problem, it is extremal because if $f:\bG\to\bD$ is any function solving the problem, then the holomorphic function $\tilde{f}:\mathbb{D}\rightarrow{\mathbb{D}}$ defined as
\begin{align}\label{Sol-f}
\tilde{f}(z)= f(2z, z^2)
\end{align}
interpolates the data $\{(0, 0), (1/2, 1/2)\}$ in $\mathbb{D}$. Thus, by Schwarz lemma we have
\begin{align}\label{DiskSol}
\tilde{f}(z)=f(2z, z^2)=z \text { for all } z\in\mathbb{D}.
\end{align} Therefore $f$ has the supremum norm over $\bG$ one.

Since the Pick problem does not have a unique solution, it is worthwhile to compute its uniqueness variety. We show that the uniqueness variety corresponding to this Pick problem is the royal variety
\begin{align}\label{UniVar}
\cR=\{(2z, z^2): z\in\mathbb{C}\}.
\end{align}
Indeed, the analysis above shows that any two interpolants coincide on the set $\{(2z, z^2): z\in\mathbb{D}\}$. This shows that $\cR$ is contained in the uniqueness set. Conversely, if $(s,p)$ is in the uniqueness set, then, in particular, the two interpolants $\frac{s}{2}$ and $-\frac{2p-s}{2-s}$ would agree at $(s,p)$, which holds if and only if $s^2=4p$. Thus the uniqueness set must be contained in
$$\{(s,p)\in\bG:s^2=4p\}=\{(2z, z^2): z\in\mathbb{D}\}.$$ Therefore the uniqueness variety must be $\cR$.
\end{example}
\begin{example}\label{E:Ill}
The data ${((0,0),0), ((0,1/2),1/2)}$ is extremal and its uniqueness variety is given by $\{(s,p):s=0\}$. If $f$ is any solution, then $\tilde{f}:\mathbb{D}\rightarrow\mathbb{D}$ given as $\tilde{f}(z)= f(0, z)$ solves the Pick data $\{(0,0),(1/2,1/2)\}$ in $\bD$. Thus by the Schwarz lemma, we have
\begin{align}\label{UniVAr2}
\tilde{f}(z)=f(0,z)=z \text { for all } z\in\mathbb{D}.
\end{align}
This shows that the problem is extremal. It is evident from \eqref{UniVAr2} that any solution of the given data has to agree on the set $\{(0, z):z\in\mathbb{D}\}$=$\{(s,p)\in\bG:s=0\}$. For the other containment, note that both the functions $p$ and $(2p-s)/(2-s)$ solve the data. They agree at $(s,p)\in\bG$ if and only if $s=0$. This establishes the claim.
\end{example}

\textbf{Funding:} The first author is supported by the Mathematical Research Impact Centric Support (MATRICS) grant, File No: MTR/2021/000560, by the Science and Engineering Research Board (SERB), Department of Science and Technology (DST),
Government of India. The second author is partially supported by a PIMS postdoctoral fellowship. The research of the third author is supported by DST INSPIRE Faculty Fellowship DST/INSPIRE/04/2018/002458. 

\textbf{Acknowledgement:} We are immensely thankful to the referees for various insightful comments; in particular to the referee who conjectured Lemma \ref{L:Added}.

\end{document}